\documentclass[11pt,reqno]{amsart}
\usepackage{amscd,amsthm, amssymb, amsmath,graphicx, tikz}

\newcommand{\p}{\partial}
\newcommand{\lk}{\ell k\,}
\newcommand{\rk}{\operatorname{rk}}
\newcommand{\Z}{\mathbb{Z}}
\newcommand{\C}{\mathbb{C}}

\newcommand{\ML}{\mathcal{L}}
\newcommand{\MA}{\mathcal{A}}
\newcommand{\MB}{\mathcal{B}}
\newcommand{\MG}{\mathcal{G}}
\newcommand{\MCR}{\mathcal{R}}
\newcommand{\Ad}{\operatorname{Ad}}
\newcommand{\ad}{\operatorname{ad}}

\newcommand{\diag}{\operatorname{diag}}
\newcommand{\Hom}{\operatorname{Hom}}
\newcommand{\R}{\mathbb{R}}

\newcommand{\OF}{\pi_1^{V}(\Sigma, \ML)}

\newcommand{\PD}{\operatorname{PD}}
\newcommand{\im}{\operatorname{im}}
\newcommand{\RP}{{\mathbb R\rm P}}
\newcommand{\Int}{\operatorname{int}}
\newcommand{\I}{I^{\natural}}
\newcommand{\IC}{IC^{\natural}}
\renewcommand{\k}{k^{\natural}}
\newcommand{\KN}{K^{\natural}}
\newcommand{\gr}{\operatorname{gr}\,}
\newcommand{\sign}{\operatorname{sign}}
\newcommand{\ind}{\operatorname{ind}}
\newcommand{\DA}{\mathcal{D}_{A}\,}
\newcommand{\DAP}{\mathcal{D}_{A'}\,}
\newcommand{\WC}{W^{\circ}}
\newcommand{\PR}{\mathcal{PR}}

\newcommand{\tr}{\operatorname{tr}}
\newcommand{\cs}{\mathbf{cs}}
\newcommand{\coker}{\operatorname{coker}}
\newcommand{\ch}{\operatorname{ch}}
\renewcommand{\Re}{\operatorname{Re}}
\newcommand{\D}{\mathcal D}
\newcommand{\so}{\mathfrak{so}}
\newcommand{\Fix}{\operatorname{Fix}\,}
\newcommand{\pistar}{\pi^*\hspace{-0.02in}}
\newcommand{\lcm}{\operatorname{lcm}}

\newtheorem{theorem}{Theorem}[section]
\newtheorem*{theorem1}{Theorem}
\newtheorem{prop}[theorem]{Proposition}
\newtheorem{lemma}[theorem]{Lemma}
\newtheorem{cor}[theorem]{Corollary}
\theoremstyle{definition}
\newtheorem{remark}[theorem]{Remark}

\newtheorem{example}[theorem]{Example}

\title{Link homology and equivariant gauge theory}
\thanks{Both authors were partially supported by NSF Grant 1065905.}
\author[Prayat Poudel]{Prayat Poudel}
\address{Department of Mathematics, University of Miami, Coral Gables, FL 33124}
\email{\rm{p.poudel@math.miami.edu}}
\author[Nikolai Saveliev]{Nikolai Saveliev}
\address{Department of Mathematics, University of Miami, Coral Gables, FL 33124}
\email{\rm{saveliev@math.miami.edu}}


\begin{document}
\begin{abstract}
The singular instanton Floer homology was defined by Kronheimer and Mrowka in connection with their proof that the Khovanov homology is an unknot detector. We study this theory for knots and two-component links using equivariant gauge theory on their double branched covers. We show that the special generator in the singular instanton Floer homology of a knot is graded by the knot signature mod 4, thereby providing a purely topological way of fixing the absolute grading in the theory. Our approach also results in explicit computations of the generators and gradings of the singular instanton Floer chain complex for several classes of knots with simple double branched covers, such as two-bridge knots, torus knots, and Montesinos knots, as well as for several families of two-components links.
\end{abstract}

\maketitle

\section{Introduction}
This paper studies the Floer homology $I_*(\Sigma,\ML)$ of two-component links $\ML \subset \Sigma$ in homology $3$-spheres defined by Kronheimer and Mrowka \cite{KM:khovanov} using singular $SO(3)$ instantons. An important special case of this theory is the singular instanton knot Floer homology $\I(k)$ for knots $k \subset S^3$ obtained by applying $I_*(S^3,\ML)$ to the link $\ML$ which is a connected sum of $k$ with the Hopf link. The Floer homology $I_*(\Sigma,\ML)$ has a relative $\Z/4$ grading, which can be upgraded to an absolute $\Z/4$ grading in the special case of $\I(k)$. Kronheimer and Mrowka \cite{KM:khovanov} used $\I(k)$ and its close cousin $I^{\sharp}(k)$ to prove that the reduced Khovanov homology is an unknot-detector.

The definition of groups $I_*(\Sigma,\ML)$ uses singular gauge theory, which makes them difficult to compute. We propose a new approach to these computations which uses equivariant gauge theory in place of the singular one. Given a two-component link $\ML$ in an integral homology sphere $\Sigma$, we pass to the double branched cover $M \to \Sigma$ with branch set $\ML$ and observe that  the singular connections on $\Sigma$ used in the definition of $I_*(\Sigma,\ML)$ pull back to equivariant smooth connections on $M$. The generators of the Floer chain complex $IC_*(\Sigma,\ML)$, whose homology is $I_*(\Sigma,\ML)$, are then derived from the equivariant representations $\pi_1 M \to SO(3)$, and their Floer gradings are computed using equivariant rather than singular index theory\footnote[2]{Note that the theory $I_*(\Sigma,\ML)$ is different from $\I (\Sigma,\ML)$ studied in \cite{KM:khovanov}: the latter is a Floer homology of a three-component link obtained by summing $\ML$ with the Hopf link.}.

As our first application of this approach, we determine the grading of the special generator in the Floer chain complex $\IC(k)$ of a knot $k \subset S^3$, see Section \ref{S:special}. This fixes the absolute $\Z/4$ grading on $\I(k)$ and confirms the conjecture of Hedden, Herald and Kirk \cite{HHK}.

\begin{theorem1}
For any knot $k \subset S^3$, the grading of the special generator in the Floer chain complex $\IC(k)$ equals $\sign k$ mod 4.
\end{theorem1}

We also achieve significant simplifications in computing the Floer chain complexes $\IC(k)$ and $IC_*(\Sigma,\ML)$ for knots and links with simple double branched covers, such as torus and Montesinos knots and links, whose double branched covers are Seifert fibered manifolds. Explicit calculations for these knots and links are possible because the gauge theory on Seifert fibered manifolds is sufficiently well developed, see Fintushel and Stern \cite{FS} and, in the equivariant setting, Collin--Saveliev \cite{CS} and Saveliev \cite{S:jdg}. Here are sample results of our calculations:

\bigskip\noindent (1)\;
The Floer chain complex $\IC(k)$ of a two-bridge knot $k$ is calculated in Section \ref{S:bridge}. For example, the Floer chain complex of the figure-eight knot consists of free abelian groups of ranks $(1,1,2,1)$. In fact, the Kronheimer--Mrowka \cite{KM:khovanov} spectral sequence is known to collapse for all two-bridge knots $k$, which implies that $\IC(k) = \I(k)$ for all such knots.

\bigskip\noindent (2)\;
The Floer chain complex $\IC(k)$ of a Montesinos knot $k = k(p,q,r)$ whose double branched cover is a Brieskorn homology sphere $\Sigma(p,q,r)$ consists of free abelian groups of ranks $(1 + b, b, b, b)$, where $b$ equals $-2$ times the Casson invariant of $\Sigma(p,q,r)$, see Section \ref{S:mont1}. General Montesinos knots are discussed in Section \ref{S:mont2}.

\bigskip\noindent (3)\;
The Floer chain complex $IC_*(S^3,\ML)$ of two-component Montesinos links $\ML = K((a_1,b_1),\ldots,(a_n,b_n))$ whose double branched cover is a homology $S^1 \times S^2$ is calculated in Section \ref{S:mont}. For example, the chain complex of the pretzel link $\ML = P(2,-3,-6)$ consists of free abelian groups of ranks $(2,0,2,0)$ up to cyclic permutation, see Section \ref{S:pretzel}. It has zero differential hence $IC_*(S^3,\ML) = I_*(S^3,\ML).$ 

\bigskip\noindent (4)\;
Our calculations for torus knots are less satisfactory because the equivariant index theory in this setting is less well developed. For instance, we prove that the Floer chain complex $\IC(k)$ of a torus knot $k = T_{p,q}$ with odd co-prime integers $p$ and $q$ has rank $1 + 4a$, where $a = -\sign\,(T_{p,q})/4$, and we conjecture that the Floer chain groups have ranks $(1 + a, a, a, a)$, see Section \ref{S:torus1}. A complete calculation of the Floer chain complex of the torus knot $T_{3,4}$ can be found in Example \ref{E:34}.

\bigskip
Some of the above results concerning two-bridge and torus knots were obtained earlier by Hedden, Herald, and Kirk \cite{HHK} using pillowcase techniques, which are completely different from our equivariant methods. We do not discuss the more difficult problem of computing the boundary operators in the Floer chain complexes $\IC(k)$ and $IC_*(\Sigma,\ML)$. Such calculations are still out of reach except in a few special cases. However, it may be worth investigating if our equivariant techniques can shed some light on this problem.

Here is an outline of the paper. It begins with a sketch of the definition of $I_*(\Sigma,\ML)$ mainly following Kronheimer and Mrowka \cite{KM:khovanov} but using the language of projective representations developed in Ruberman--Saveliev \cite{RSI}, see also Dostoglou--Salamon \cite{DS}. We obtain a purely algebraic description of the generators in $IC_*(\Sigma,\ML)$ as well as of a certain natural $\Z/2\,\oplus\,\Z/2$ action on them, which is crucial to the rest of the paper.

Equivariant gauge theory is developed in Section \ref{S:equiv}. The section begins with a computation of $\Z/2$ cohomology rings of double branched covers $M \to \Sigma$ of two-component links, followed by a computation of the characteristic classes of $SO(3)$ bundles on $M$ pulled back from orbifold bundles on $\Sigma$. The results are used to establish a bijective correspondence between equivariant $SO(3)$ representations of $\pi_1 M$ and orbifold $SO(3)$ representations of $\pi_1 \Sigma$. In the rest of the section, we discuss equivariant index theory which is used later in the paper to compute Floer gradings of the generators in $IC_*(\Sigma,\ML)$. Our equivariant index theory approach is also used to recover the Kronheimer--Mrowka \cite{KM:khovanov} singular index formulas along the lines of Wang's paper \cite{wang}.

The next five sections are dedicated to the singular knot Floer homology $\I (k)$ for knots $k \subset S^3$. Section \ref{S:gen} describes generators in the chain complex $\IC (k)$ in terms of equivariant representations $\pi_1 Y \to SO(3)$ on the double branched cover $Y \to S^3$ with branch set the knot $k$. These representations fall into three categories: trivial, reducible non-trivial, and irreducible. 

The trivial representation $\theta: \pi_1 Y \to SO(3)$ gives rise to a special generator $\alpha \in \IC(k)$ which is used in \cite{KM:khovanov} to fix an absolute grading on $\I(k)$. This generator is dealt with in Section \ref{S:special}. We pass to the double branched cover and use Taubes \cite{T:periodic} index theory on manifolds with periodic ends to show that the Floer grading of $\alpha$ equals $\sign\,(k)$ mod 4.

Having computed the absolute index of $\alpha$, we only need to compute the relative indices of the remaining generators. We derive formulas for these gradings in Section \ref{S:other} using equivariant index calculations on double branched covers, and apply these formulas to Montesinos and torus knots in the section that follows.

Section \ref{S:links} contains calculations of $IC_*(\Sigma,\ML)$ for several two-component links $\ML$ not of the form $\k$. For the pretzel link $\ML = P(2,-3,-6)$ in the 3-sphere we obtain a complete calculation of the Floer homology groups of $P(2,-3,-6)$ and not just of the Floer chain complex. The same answer is independently confirmed by computing the Floer homology of Harper--Saveliev \cite{HS} for this two-component link: the latter theory is isomorphic to $I_*(\Sigma,\ML)$ but does not use singular connections in its definition.

Finally, Section \ref{S:top} contains proofs of some topological results, which were postponed earlier in the paper for the sake of exposition.

\smallskip\noindent
\textbf{Acknowledgments:} We are thankful to Ken Baker, Paul Kirk, and Daniel Ruberman for useful discussions.


\section{Link homology}\label{S:km}
In this section, we will sketch the definition of the singular instanton Floer homology $I_*(\Sigma,\ML)$ of a two-component link $\ML \subset \Sigma$ in an integral homology $3$-sphere. We will follow Kronheimer and Mrowka \cite{KM:khovanov} closely, deviating in just two respects: we will use the language of projective representations to describe the generators in the Floer chain complex, and will introduce a canonical $\Z/2\,\oplus\,\Z/2$ action on these generators.


\subsection{The Chern--Simons functional}
Given a two-component link $\ML$ in an integral homology sphere $\Sigma$, the second homology of its exterior $X = \Sigma - \Int N(\ML)$ is isomorphic to a copy of $\Z$ spanned by either one of the boundary tori of $X$. Let $P \to X$ be the unique $SO(3)$ bundle with a non-trivial  second Stiefel--Whitney class $w_2 (P) \in H^2 (X; \Z/2) = \Z/2$. The flat connections in this bundle serve as the starting point for building $I_*(\Sigma,\ML)$. Since $w_2 (P)$ evaluates non-trivially on the boundary tori, these connections are necessarily irreducible and have order two holonomy along the meridians of the link components. Therefore, they give rise to flat connections in an orbifold $SO(3)$ bundle on $\Sigma$, which we again call $P$. The homology sphere $\Sigma$ itself is viewed as an orbifold with the orbifold singularity $\ML$, equipped with a Riemannian metric with cone angle $\pi$ along the singular set.

Kronheimer and Mrowka \cite{KM:khovanov} interpreted the gauge equivalence classes of the orbifold flat connections in $P$ as the critical points of an orbifold Chern--Simons functional
\begin{equation}\label{E:cs}
\cs: \MB\, (\Sigma,\ML) \to \R/\Z,
\end{equation}
and defined $I_*(\Sigma,\ML)$ as its Morse homology.  An important feature of this construction is the use of the restricted orbifold gauge group $\MG_S$ in the definition of the configuration space,
\[
\MB\,(\Sigma,\ML) = \MA\,(\Sigma,\ML)/\MG_S,
\]
where $\MA\,(\Sigma,\ML)$ is an affine space of orbifold connections and $\MG_S$ is the quotient of the determinant-one orbifold gauge group $\MG(\check P)$ of Kronheimer and Mrowka \cite[Section 2.6]{KM:khovanov} by its center $\{\pm 1\}$. The group $\MG_S$ is a normal subgroup of the full orbifold gauge group $\MG$ with the quotient $\MG/\MG_S = H^1 (X;\Z/2) = \Z/2\,\oplus\,\Z/2$. The full gauge group $\MG$ acts on $\MA\,(\Sigma,\ML)$ preserving the gradient of $\cs$, thereby giving rise to the residual action of $H^1 (X;\Z/2)$ on the configuration space $\MB\,(\Sigma,\ML)$ and on the critical point set of the Chern--Simons functional.

We will next describe the critical points of the functional \eqref{E:cs} algebraically using the holonomy correspondence between flat connections and representations of the fundamental group. A variant of this classical correspondence which applies to the situation at hand was described in \cite[Section 3.2]{RSI} using projective $SU(2)$ representations. We will review these first, see \cite[Section 3.1]{RSI} for details.


\subsection{Projective representations}\label{S:proj}
Let $G$ be a finitely presented group and view the center of $SU(2)$ as $\Z/2 = \{ \pm 1 \}$. A map $\rho: G \to SU(2)$ is called a projective representation if 
\[
c(g,h)\, =\, \rho(gh)\, \rho(h)^{-1} \rho(g)^{-1} \in \Z/2\quad\text{for all}\;\; g,h \in G. 
\]
The function $c: G \times G \to \Z/2$ is a 2-cocycle on $G$ defining a cohomology class $[c] \in H^2 (G;\Z/2)$. This class has the following interpretation. The composition of $\rho: G \to SU(2)$ with $\Ad: SU(2) \to SO(3)$ is a representation $\Ad \rho: G \to SO(3)$. As such, it induces a continuous map $BG \to BSO(3)$ which is unique up to homotopy. The pull back of the universal Stiefel--Whitney class $w_2 \in H^2 (BSO(3);\Z/2)$ via this map is our class $[c] = w_2 (\Ad\rho) \in H^2 (G; \Z/2)$. It serves as an obstruction to lifting $\Ad\rho: G \to SO(3)$ to an $SU(2)$ representation.

Let $\PR_c (G; SU(2))$ be the space of conjugacy classes of projective representations $\rho: G \to SU(2)$ whose associated cocycle is $c$. The topology on $\PR_c (G; SU(2))$ is supplied by the algebraic set structure. One can easily see that $\PR_c (G; SU(2))$ is determined uniquely up to homeomorphism by the cohomology class of $c$. The group $H^1(G;\Z/2) = \Hom(G, \Z/2)$ acts on $\PR_c (G;SU(2))$ by sending $\rho$ to $\chi \cdot \rho$ for any $\chi \in \Hom(G, \Z/2)$. The orbits of this action are in a bijective correspondence with the conjugacy classes of representations $G \to SO(3)$ whose second Stiefel--Whitney class equals $[c]$. The bijection is given by taking the adjoint representation.

Projective representations $\rho: G \to SU(2)$ can also be described in terms of a presentation $G = F/R$. Consider a homomorphism $\gamma: R \to \Z/2$ defined by its values $\gamma(r) = \pm 1$ on the relators $r \in R$ and by the condition that it is constant on the orbits of the adjoint action of $F$ on $R$. Also, choose a set-theoretic section $s: G \longrightarrow F$ in the exact sequence
\[
\begin{tikzpicture}
\draw (1.5,1) node (a) {$1$};
\draw (3,1) node (b) {$R$};
\draw (5,1) node (c) {$F$};
\draw (7,1) node (d) {$G$};
\draw (8.5,1) node (e) {$1$};
\draw[->](a)--(b);
\draw[->](b)--(c) node [midway,above](TextNode){$i$};
\draw[->](c)--(d) node [midway,above](TextNode){$\pi$};
\draw[->](d)--(e);
\end{tikzpicture}
\]
and denote by $r: G \times G \longrightarrow R$ the function defined by the formula $s(gh) = r(g,h) s(g)s(h)$.

\begin{prop}\label{P:proj}
A choice of a section $s: G \to F$ establishes a bijective correspondence between the conjugacy classes of projective representations $\rho: G \to SU(2)$ with the cocycle $c(g,h) = \gamma(r(g,h))$,  and the conjugacy classes of homomorphisms $\sigma: F \to SU(2)$ such that $i^*\sigma = \gamma$. A different choice of $s$ results in a cohomologous cocycle.
\end{prop} 

\begin{proof}
We begin by checking that $c(g,h) = \gamma(r(g,h))$ is a cocycle. For any $g, h, k \in G$, we have
\begin{gather}
s(ghk) = r(gh,k) s(gh) s(k) = r(gh,k) r(g,h) s(g) s(h) s(k), \notag \\
s(ghk) = r(g,hk) s(g) s(hk) = r(g,hk) s(g) r(h,k) s(h) s(k), \notag
\end{gather}
which results in $r(gh,k) r(g,h) = r(g,hk) s(g) r(h,k) s(g)^{-1}$. Since the homomorphism $\gamma$ is constant on the orbits of the adjoint action of $F$ on $R$, its application to the above equality gives the cocycle condition $c(gh,k) c(g,h) = c(g,hk) c(h,k)$ as desired.

Now, given a homomorphism $\sigma: F \to SU(2)$ such that $i^*\sigma = \gamma$, define $\rho: G \to SU(2)$ by the formula $\rho(g) = \sigma(s(g))$. Then $\rho(gh) = \sigma(s(gh)) = \sigma(r(g,h)s(g)s(h)) = \gamma(r(g,h)) \sigma(s(g)) \sigma(s(h)) = c(g,h) \rho(g)\rho(h)$, hence $\rho$ is a projective representation with cocycle $c$. It is clear that conjugate representations $\sigma$ define conjugate projective representations $\rho$, and that a different choice of $s$ leads to a cohomologous cocycle $c$.

The inverse correspondence is defined as follows. Given a projective representation $\rho: G \to SU(2)$, write elements of $F$ in the form $r\cdot s(g)$ with $r \in R$ and $g \in G$, and define $\sigma: F \to SU(2)$ by the formula $\sigma(r\cdot s(g)) = \gamma(r) \rho(g)$. That $\sigma$ is a homomorphism can be checked by a straightforward calculation using the fact that $c(g,h) = \gamma(r(g,h))$.
\end{proof}

\begin{example}\label{E:floer}
Let $G = \pi_1 M$ be the fundamental group of a manifold $M$ obtained by 0--surgery on a knot $k$ in an integral homology sphere $\Sigma$. The group $\pi_1 M$ is obtained from $\pi_1 K$ by imposing the relation $\lambda = 1$, where $\lambda$ is a canonical longitude of $k$. Therefore, $\pi_1 M$ admits a presentation $\pi_1 M = F/R$ with $\lambda$ being one of the relators. Let $\gamma (\lambda) = -1$ and $\gamma(r) = 1$ for the rest of the relators $r \in R$. It has been known since Floer \cite{floer:triangle} that the action of $H^1(M;\Z/2) = \Z/2$ on the set of conjugacy classes of projective representations $\sigma: F \to SU(2)$ with $i^*\sigma = \gamma$ is free, providing a two-to-one correspondence between this set and the set of the conjugacy classes of representations $\pi_1 M \to SO(3)$ with non-trivial $w_2 \in H^2 (M;\Z/2) = \Z/2$.
\end{example}


\subsection{Holonomy correspondence}\label{S:hol}
We will now apply the general theory of Section \ref{S:proj} to the group $G = \pi_1 X$, where $X$ is the exterior of a two-component link $\ML$ in an integral homology sphere $\Sigma$. We begin with the following simple observation.

\begin{lemma}\label{L:hopf}
Unless the link $\ML$ is split, $H^2 (X;\Z/2) = H^2 (\pi_1 X;\Z/2) =\Z/2$. For split links, $I_*(\Sigma,\ML) = 0$.
\end{lemma}

\begin{proof}
For a split link $\ML$, the splitting sphere generates the group $H_2 (X;\Z) = \Z$. Since there are no flat connections on this sphere with non-trivial $w_2 (P)$ the group $I_* (\Sigma, \ML)$ must vanish. For a non-split link, the claimed equality follows from the Hopf exact sequence

\[
\begin{tikzpicture}
\draw (1,1) node (a) {$\pi_2(X)$};
\draw (3,1) node (b) {$H_2(X)$};
\draw (6,1) node (c) {$H_2(\pi_1 X)$};
\draw (8,1) node (d) {$0$};
\draw[->](a)--(b);
\draw[->](b)--(c);
\draw[->](c)--(d);
\end{tikzpicture}
\]
and the vanishing of the Hurewicz homomorphism $\pi_2 (X) \to H_2 (X)$.
\end{proof}

From now on, we will assume that the link $\ML \subset \Sigma$ is not split. The holonomy correspondence of \cite[Section 3.1]{RSI} identifies the critical point set of the functional \eqref{E:cs} with the set $\PR_c (X,SU(2))$ of the conjugacy classes of  projective representations $\rho: \pi_1 X \to SU(2)$, for any choice of cocycle $c$ such that $0 \ne [c] = w_2 (P) \in H^2 (X;\Z/2) = \Z/2$. Note that this identification commutes with the $H^1 (X;\Z/2)$ action, and that the orbits of this action on $\PR_c (X,SU(2))$ are in  a bijective correspondence with the conjugacy classes of representations $\Ad \rho: \pi_1 X \to SO(3)$ having $w_2 (\Ad\rho) \ne 0$.

\begin{lemma}\label{L:irr}
Any representation $\Ad \rho: \pi_1 X \to SO(3)$ with $w_2 (\Ad\rho) \ne 0$ is irreducible, that is, its image is not contained in a copy of $SO(2) \subset SO(3)$.
\end{lemma}

\begin{proof}
The restriction to $\rho$ to either boundary torus of $X$ has non-trivial second Stiefel--Whitney class, which implies that it does not lift to an $SU(2)$ representation. However, any reducible representation $\pi_1 T^2 \to SO(3)$ admits an $SU(2)$ lift, therefore, the image of $\rho$ cannot be contained in a copy of $SO(2) \subset SO(3)$. It is essential here that $H_1 (T^2)$ has no 2-torsion: a non-trivial $SO(3)$ representation of $\Z/2$ is reducible but does not admit an $SU(2)$ lift.
\end{proof}


\subsection{Floer gradings}
Given flat orbifold connections $\rho$ and $\sigma$ in the orbifold bundle $P \to \Sigma$, consider an arbitrary orbifold connection $A$ in the pull back bundle on the product $\R\,\times\,\Sigma$ matching $\rho$ and $\sigma$ near the negative and positive ends, respectively. Equip $\R\,\times\,\Sigma$ with the orbifold product metric and consider the ASD operator 
\begin{equation}\label{E:DA}
\DA(\rho,\sigma) = -d^*_A\,\oplus\,d^+_A: \;\Omega^1 (\R \times \Sigma, \ad P) \to (\Omega^0\,\oplus\,\Omega^2_+) (\R \times \Sigma,\ad P)
\end{equation}
completed in the orbifold Sobolev $L^2$ norms as in \cite[Section 3.1]{KM:khovanov}. Since $\rho$ and $\sigma$ are irreducible, this operator will be Fredholm if we further assume that $\rho$ and $\sigma$ are non-degenerate as the critical points of the Chern--Simons functional \eqref{E:cs}. Define the relative Floer grading as
\begin{equation}\label{E:rel-gr}
\gr(\rho,\sigma)\; =\; \ind \DA(\rho,\sigma)\;\;\text{(mod 4)}. 
\end{equation}
This grading is well defined because replacing either $\rho$ or $\sigma$ by its gauge equivalent within the restricted gauge group $\MG_S$ results in adding a multiple of four to the index of $\DA$, see \cite[Section 2.5]{KM:khovanov}. This is no longer true if we use the full gauge group. The following lemma makes it precise; it will be proved in Section \ref{S:deg}.

\begin{lemma}\label{L:deg}
Let $\chi_1$ and $\chi_2$ be the generators of $H^1 (X;\Z/2) = \Z/2\,\oplus\,\Z/2$ dual to the meridians of the link $\ML = \ell_1\,\cup\,\ell_2$. Then 
\[
\gr(\chi_1 \cdot \rho,\sigma)\, =\, \gr(\chi_2 \cdot\rho,\sigma)\, =\, \gr (\rho,\sigma)\, +\, 2\cdot\delta
\pmod 4,
\]
and similarly for the action on $\sigma$, where
\[
\delta\; =\; 
\begin{cases}
\; 0, &\text{if\; $\lk (\ell_1,\ell_2)$\, is odd}, \\
\; 1, &\text{if\; $\lk (\ell_1,\ell_2)$\, is even}.
\end{cases}
\]
\end{lemma}


\subsection{Perturbations}
The critical points of the Chern--Simons functional need not be non-degenerate, therefore, we may have to perturb it to define $I_*(\Sigma,\ML)$. The perturbations used in \cite[Section 3.4]{KM:khovanov} are the standard Wilson loop perturbations along loops in $\Sigma$ disjoint from the link $\ML$. There are sufficiently many such perturbations to guarantee the non-degeneracy of the critical points of the perturbed Chern--Simons functional as well as the transversality properties for the moduli spaces of trajectories of its gradient flow. This allows to define the boundary operator and to complete the definition of $I_*(\Sigma,\ML)$.


\section{Equivariant gauge theory}\label{S:equiv}
In this section, we survey some equivariant gauge theory on the double branched cover $M \to \Sigma$ of a homology sphere $\Sigma$ with branch set a two-component link $\ML$. It will be used in the forthcoming sections to make headway in computing the link homology $I_*(\Sigma,\ML)$.


\subsection{Topological preliminaries}\label{S:prelim}
Let $\Sigma$ be an integral homology 3-sphere and $\ML = \ell_1\cup\ell_2$ a link of two components in $\Sigma$. The link exterior $X = \Sigma - \Int N(\ML)$ is a manifold whose boundary consists of two tori, with $H_1 (X;\Z) = \Z^2$ spanned by the meridians $\mu_1$ and $\mu_2$ of the link components. The homomorphism $\pi_1 X \to \Z/2$ sending $\mu_1$ and $\mu_2$ to the generator of $\Z/2$ gives rise to a regular double cover $\widetilde X \to X$, and also to a double branched cover $\pi: M \to \Sigma$ with branching set $\ML$ and the covering translation $\tau: M \to M$. Denote by $\Delta(t)$ the one-variable Alexander polynomial of $\ML$.

\begin{prop}\label{P:kawauchi}
 The first Betti number of $M$ is one if $\Delta(-1) = 0$ and zero otherwise. In the latter case, $H_1 (M;\Z)$ is a finite group of order $|\Delta(-1)|$. The induced involution $\tau_*: H_1 (M) \to H_1 (M)$ is multiplication by $-1$.
\end{prop}

\begin{proof}
This is essentially proved in Kawauchi \cite[Section 5.5]{K}. The statement about $\tau_*$ follows from an isomorphism of $\Z[t,t^{-1}]$ modules $H_1 (M) = H_1(E)/(1+t) H_1(E)$, where $E$ is the infinite cyclic cover of $X$, established in \cite[Theorem 5.5.1]{K}. A completely different proof for the special case of double branched covers of $S^3$ with branch set a knot can be found in Ruberman \cite[Lemma 5.5]{R:slice}.
\end{proof}

\begin{prop}\label{P:H2}
Let $M$ be the double branched cover of an integral homology sphere with branch set a two-component link. Then $H_i (M;\Z/2) = H^i (M;\Z/2)$ is isomorphic to $\Z/2$ if $i = 0, 1, 2, 3$, and is zero otherwise. The cup-product $H^1(M;\Z/2) \times H^1 (M;\Z/2) \to H^2 (M;\Z/2)$ is given by the linking number $\lk(\ell_1,\ell_2)\pmod 2$.
\end{prop}

The proof of Proposition \ref{P:H2} will be postponed until Section \ref{S:top} for the sake of exposition. 

An important example of $\ML$ to consider is that of the two-component link $\k$ obtained by connect summing a knot $k \subset S^3$ with the Hopf link. The double branched cover $M \to S^3$ in this case is the connected sum $M = Y \,\#\;\RP^3$, where $Y$ is the double branched cover of $k$. Proposition \ref{P:H2} easily follows because $H_*(Y;\Z/2) = H_*(S^3;\Z/2)$.


\subsection{The orbifold exact sequence}\label{S:orb}
We will view $\Sigma = M/\tau$ as an orbifold with the singular set $\ML$. To be precise, the regular double cover $\widetilde X \to X$ is a $3$-manifold whose boundary consists of two tori, and
\[
M = \widetilde X\,\cup_h\,N(\ML),
\]
where the gluing homeomorphism $h: \partial \widetilde X \to \partial N(\ML)$ identifies $\pi^{-1}(\mu_i)$ with the meridian $\mu_i$ for $i = 1, 2$. The involution $\tau: M \to M$ acts by meridional rotation on $N(\ML)$, thereby fixing the link $\ML$, and by covering translation on $\widetilde X$. Define the orbifold fundamental group
\begin{equation*}
\OF = \pi_1 X\,\big/\,\langle {\mu_1}^2 = {\mu_2}^2 = 1\rangle
\end{equation*}
then the homotopy exact sequence of the covering $\widetilde X \to X$ gives rise to a split short exact sequence, called the orbifold exact sequence,
\medskip
\begin{equation}\label{E:orb}
\begin{CD}
1 @>>> \pi_1 M @> \displaystyle{\pi_*} >> \pi_1^{V}(\Sigma, \mathcal{L}) @> \displaystyle{j} > \empty > \Z/2 @>>> 1.
\end{CD}
\end{equation}

\medskip\noindent
The homomorphism $j$ maps the meridians $\mu_1, \mu_2$ to the generator of $\Z/2$ and one obtains a splitting by sending this generator to either $\mu_1$ or $\mu_2$. 

It follows from the definition of the orbifold fundamental group $\OF$ that its abelianization equals $H_1(X)\,\big/\,\langle {\mu_1}^2 = {\mu_2}^2 = 1\rangle = H_1 (X;\Z/2) = \Z/2\,\oplus\,\Z/2$ with the canonical generators $\mu_1$ and $\mu_2$. The homomorphism $\pi_*$ of the orbifold exact sequence \eqref{E:orb} then induces a map $\pi_*: H_1 (M;\Z/2) \to H_1 (X;\Z/2)$ which can be described as follows.

\begin{lemma}\label{L:branched}
The homomorphism $\pi_*: H_1 (M;\Z/2) \to H_1 (X;\Z/2)$ sends the generator of $H_1 (M;\Z/2) = \Z/2$ to the sum of the meridians $\mu_1 + \mu_2 \in H_1 (X;\Z/2)$.
\end{lemma}

\begin{proof}
That $H_1 (M;\Z/2) = \Z/2$ follows from Proposition \ref{P:H2}. An explicit generator of this group is described in the proof of Proposition \ref{P:cup} as the circle $\pi^{-1} (w)$, where $w$ is an embedded arc in $\Sigma$ with endpoints on the two different components of $\ML$. The commutative diagram
\begin{equation*}
\begin{tikzpicture} [scale =1]
\draw(1,1) node(a) {$\pi_1 M$ };
\draw(5,1) node(b) {$\OF$};
\draw(1,3) node(c) {$\pi_1 \widetilde X$};
\draw(5,3) node(d) {$\pi_1 X$};
\draw[->](a)--(b) node [midway,above](TextNode){$\pi_*$};
\draw[->](c)--(d) node [midway,above](TextNode){$\pi_*$};
\draw[->](c)--(a) ; 
\draw[->](d)--(b) ;
\end{tikzpicture}
\end{equation*}
gives rise to the commutative diagram in homology
\begin{equation*}
\begin{tikzpicture} [scale =1]
\draw(1,1) node(a) {$H_1(M;\Z/2)$ };
\draw(5,1) node(b) {$H_1(X;\Z/2)$};
\draw(1,3) node(c) {$H_1(\widetilde X;\Z/2)$};
\draw(5,3) node(d) {$H_1(X;\Z/2)$};
\draw[->](a)--(b) node [midway,above](TextNode){$\pi_*$};
\draw[->](c)--(d) node [midway,above](TextNode){$\pi_*$};
\draw[->](c)--(a) ; 
\draw[->](d)--(b) ; 
\end{tikzpicture}
\end{equation*}
The cycle $\pi^{-1} (w)$ in $M$ is homologous to a cycle in $\widetilde X$ which consists of the two arcs $\pi^{-1} (w)\,\cap\, \widetilde X$ whose endpoints on each of the tori in $\p \widetilde X$ are connected by an arc. The map $\pi_*: H_1 (\widetilde X;\Z/2) \to H_1 (X;\Z/2)$ takes the homology class of this cycle to $\mu_1 + \mu_2$ and the result follows.
\end{proof}


\subsection{Pulled back bundles}\label{S:bundles}
Let $P \to \Sigma$ be the orbifold $SO(3)$ bundle used in the definition of $I_*(\Sigma,\ML)$ in Section \ref{S:km}. It pulls back to an orbifold $SO(3)$ bundle $Q \to M$ because the projection map $\pi: M \to \Sigma$ is regular in the sense of Chen--Ruan \cite{CR}. The bundle $Q$ is in fact smooth because orbifold connections on $P$ with order-two holonomy along the meridians of $\ML$ lift to connections in $Q$ with trivial holonomy along the meridians of the two-component link $\tilde \ML = \pi^{-1} (\ML)$. 

\begin{prop}\label{P:Q}
The bundle $Q \to M$ is non-trivial.
\end{prop}

The rest of this section is dedicated to the proof of this proposition. We will accomplish it by showing the non-vanishing of $w_2 (Q) \in H^2 (M;\Z/2) = \Z/2$. Our argument will split into two cases, corresponding to the parity of the linking number between the components of $\ML$.

Suppose that $\lk (\ell_1,\ell_2)$ is even and consider the regular double cover $\pi: M - \widetilde{\ML} \rightarrow \Sigma - \ML$. It gives rise to the Gysin exact sequence
\smallskip
\begin{equation*}
\begin{tikzpicture}
\draw (1.2,2) node (k) {};
\draw (3.3,2) node (l) {$H^{1}(\Sigma - \ML;\Z/2)$};
\draw (7.5,2) node (m) {$H^{2}(\Sigma - \ML;\Z/2)$};
\draw (11.5,2) node (n) {$H^{2}(M - \widetilde{\ML};\Z/2)$};
\draw (13.5,2) node(o) {};
\draw[->](k)--(l);
\draw[->](l)--(m) node [midway,above](TextNode){$\cup \, w_1$};
\draw[->](m)--(n) node [midway,above](TextNode){$\pi^{*}$};
\draw[->](n)--(o);
\draw (1.2,1) node (p) {};
\draw (3.3,1) node (q) {$H^{2}(\Sigma - \ML;\Z/2)$};
\draw (7.5,1) node (r) {$H^{3}(\Sigma - \ML;\Z/2)$};
\draw (10.3,1) node (s) {$\cdots$};
\draw[->](p)--(q);
\draw[->](q)--(r) node [midway,above](TextNode){$\cup \, w_1$};
\draw[->](r)--(s);
\end{tikzpicture}
\end {equation*}
where $\cup\, w_1$ means taking the cup-product with the first Stiefel--Whitney class of the cover. The cup-product on $H^*(\Sigma-\ML;\Z/2)$ can be determined from the following commutative diagram

\begin{equation*}
\begin{tikzpicture} [scale =1]
\draw(1,1) node(a) {$H^{1}(\Sigma-\ML;\Z/2) \times H^{1}(\Sigma-\ML;\Z/2)$ };
\draw(6.5,1) node(b) {$H^{2}(\Sigma-\ML;\Z/2)$};
\draw(1,3) node(c) {$H_{2}(\Sigma,\ML;\Z/2) \times H_{2}(\Sigma,\ML;\Z/2) $};
\draw(6.5,3) node(d) {$H_{1}(\Sigma,\ML;\Z/2)$};
\draw[->](a)--(c) node [midway,left](TextNode){PD};
\draw[->](b)--(d) node [midway,right](TextNode){PD};
\draw[->](a)--(b) node [midway, above] (TextNode) {$\cup$};
\draw[->](c)--(d) node [midway, above] (TextNode) {$\cdot$};
\end{tikzpicture}
\end{equation*}

\medskip\noindent
where $\PD$ stands for the Poincar\'e duality isomorphism and the dot in the upper row for the intersection product. Note that Seifert surfaces of knots $\ell_1$ and $\ell_2$ generate $H_2 (\Sigma,\ML; \Z/2) = \Z/2\,\oplus\,\Z/2$, and any arc in $\Sigma$ with one endpoint on $\ell_1$ and the other on  $\ell_2$ generates $H_1 (\Sigma,\ML; \Z/2) = \Z/2$. An easy calculation shows that, with respect to these generators, the intersection product is given by the matrix 

\begin{equation*}
\begin{pmatrix}
0 & \lk(\ell_1, \ell_2) \\
\lk(\ell_1,\ell_2) & 0 \\
\end{pmatrix}
\end{equation*}

\medskip\noindent
Since $\lk(\ell_1,\ell_2)$ is even, this gives a trivial cup product structure on the link complement $\Sigma - \ML$. Therefore, the map $\cup\,w_1$ in the Gysin sequence is zero and the map $\pi^*: H^{2}(\Sigma - \ML;\Z/2) \to H^{2}(M - \widetilde{\ML};\Z/2)$ is injective. Since $w_2(P) \in H^2 (\Sigma - \ML;\Z/2)$ is non-zero we conclude that $\pi^{*}(w_2(P)) \neq 0$. This implies that $w_2 (Q) \neq 0$ because $Q = \pi^* P$ over $M - \widetilde\ML$.

Now suppose that $\lk(\ell_1,\ell_2)$ is odd. The above calculation implies that the second Stiefel--Whitney class of $\pi^* P$ vanishes in $H^2(M - \widetilde\ML;\Z/2)$. We will prove, however, that $w_2 (Q) \in H^2 (M;\Z/2)$ is non-zero, by showing that $Q$ carries a flat connection with non-zero $w_2$.

Note that the orbifold bundle $P$ carries a flat $SO(3)$ connection whose holonomy is a representation $\alpha : \OF \to SO(3)$ of the orbifold fundamental group $\OF = \pi_1 X/\langle{\mu_1}^2 = {\mu_2}^2 = 1\rangle$ sending the two meridians to $\Ad i$ and $\Ad j$. This flat connection pulls back to a flat connection on $Q$ with holonomy $\pi^*\alpha: \pi_1 M \to SO(3)$. We wish to compute the second Stiefel--Whitney class of $\pi^*\alpha$. 

\begin{lemma}\label{L:pi}
The representation $\pi^*\alpha : \pi_1 M \to \Z/2\,\oplus\,\Z/2$  is non-trivial.
\end{lemma} 

\begin{proof}
Our proof will rely on the orbifold exact sequence \eqref{E:orb}. Assume that $\pi^* \alpha$ is trivial. Then $\pi_1 M \subset \ker (\pi^*\alpha)$ hence $\alpha$ factors through a homomorphism $\OF / \pi_*(\pi_1 M) \to \Z/2\,\oplus\,\Z/2$. Since $\OF / \pi_*(\pi_1 M) = \Z/2$ we obtain a contradiction with the surjectivity of $\alpha$.
\end{proof}

Since the group $\Z/2\,\oplus\,\Z/2$ is abelian, the representation $\pi^*\alpha: \pi_1 M \to \Z/2\,\oplus\,\Z/2$ factors through a homomorphism $H_1(M) \to \Z/2\,\oplus\,\Z/2$ which is uniquely determined by its two components $\xi$, $\eta \in \Hom(H_1(M),\Z/2) = H^1(M;\Z/2) = \Z/2$, see Proposition \ref{P:H2}. A calculation identical to that in \cite[Proposition 4.3]{RSI} shows that $w_2(\pi^*\alpha) = \xi\,\cup\,\xi + \xi\,\cup\, \eta + \eta\,\cup\,\eta$ (note that, unlike in \cite{RSI}, the classes $\xi\,\cup\,\xi$ and $\eta\,\cup\,\eta$ need not vanish). Since $\xi$ and $\eta$ cannot be both trivial by Lemma \ref{L:pi}, we may assume without loss of generality that $\xi \neq 0$. If $\eta = 0$ then $w_2(\pi^*\alpha) = \xi\,\cup\,\xi$. If $\eta \neq 0$ then $\xi = \eta$ due to the fact that $H^{1}(M;\Z/2) = \Z/2$, and therefore again $w_2(\pi^*\alpha) = \xi\,\cup\,\xi$. Since $\lk(\ell_1,\ell_2)$ is odd, it follows from Proposition \ref{P:H2} that $w_2 (\pi^*\alpha) \neq 0$.


\subsection{Pulled back representations}
Assuming that $\ML \subset \Sigma$ is non-split, we identified in Section \ref{S:hol} the critical point set of the Chern--Simons functional \eqref{E:cs} with the space $\PR_c (X,SU(2))$ of the conjugacy classes of  projective representations $\pi_1 X \to SU(2)$ on the link exterior, for any choice of cocycle $c$ not cohomologous to zero. We further identified the quotient of $\PR_c (X,SU(2))$ by the natural $H^1 (X;\Z/2)$ action with the subspace $\MCR_w (X;SO(3))$ of the $SO(3)$ character variety of $\pi_1 X$ cut out by the condition $w_2 \neq 0$. The latter condition implies that both meridians $\mu_1$ and $\mu_2$ are represented by $SO(3)$ matrices of order two, which leads to a natural identification of this subspace with 
\[
{\MCR}_{\omega} (\Sigma, \ML ; SO(3)) = \{\,\rho: \OF \to SO(3))\; |\; w_2(\rho) \neq 0 \,\}/\Ad SO(3),
\]
where the condition $w_2(\rho) \neq 0$ applies to the representation $\rho$ restricted to $X$. To summarize, the group $H^1 (X;\Z/2)$ acts on the space $\PR_c (X,SU(2))$ with the quotient map
\[
\PR_c (X,SU(2)) \longrightarrow {\MCR}_{\omega} (\Sigma, \ML ; SO(3)).
\]
We now wish to study the space ${\MCR}_{\omega} (\Sigma, \ML ; SO(3))$ using representations on the double branched cover $M \to \Sigma$ equivariant with respect to the covering translation $\tau: M \to M$.

\begin{lemma}\label{L:orb}
Let $\rho: \OF \to SO(3)$ be a representation with $w_2 (\rho) \neq 0$, and $\pi^*\rho: \pi_1 M \to SO(3)$ its pull back via the homomorphism $\pi_*$ of the orbifold exact sequence \eqref{E:orb}. Then there exists an element $u \in SO(3)$ of order two such that $\tau^*(\pi^*\rho) = u \cdot (\pi^*\rho) \cdot u^{-1}$.
\end{lemma}

\begin{proof}
Let $\widetilde X \to X$ be the regular double cover as in Section \ref{S:orb}. Choose a basepoint $b$ in one of the boundary tori of $\widetilde X$ and consider the commutative diagram
\begin {equation*}
\begin{tikzpicture}
\draw (1,3) node (a) {$\pi_1 (\widetilde{X},b)$};
\draw (5,3) node (b) {$\pi_1(\widetilde{X},\tau(b))$};
\draw (8,3) node (e) {$\pi_1(\widetilde{X},b)$};
\draw (3,1) node (c) {$\pi_1(X,\pi(b))$};
\draw (8,1) node (d) {$\pi_1(X,\pi(b))$};
\draw[->](a)--(b) node [midway,above](TextNode){$\tau_*$};
\draw[->](a)--(c) node [midway,left](TextNode){$\pi_{*}$};
\draw[->](b)--(c) node [midway,right](TextNode){$\; \pi_{*}$};
\draw[->](e)--(d) node [midway,right](TextNode){$\pi_{*}$};
\draw[->](c)--(d) node [midway,above](TextNode){$\varphi$};
\draw[->](b)--(e) node[midway, above](TextNode){$\psi_f$};
\end{tikzpicture}
\end {equation*}
whose maps $\psi_f$ and $\varphi$ are defined as follows. Given a path $f: [0,1] \to X$ from $b$ to $\tau(b)$, take its inverse $\overline{f}(s) = f(1-s)$ and define the map $\psi_f$ by the formula $\psi_{f} (\beta) = f \cdot \beta \cdot \overline{f}$. Since $\pi(b) = \pi(\tau(b))$, the path $f$ projects to a loop in $X$ based at $\pi(b)$, and the map $\varphi$ is the conjugation by that loop. In fact, one can choose the path $f$ to project onto the meridian $\mu_i$ of the boundary torus on which $\pi(b)$ lies so that $\varphi(x) = \mu_i\cdot x \cdot \mu_i^{-1}$. After filling in the solid tori, we obtain the commutative diagram
\begin{equation*}
\begin{tikzpicture} [scale =1]
\draw(1,1) node(a) {$\OF$ };
\draw(5,1) node(b) {$\OF$};
\draw(1,3) node(c) {$\pi_1 M$};
\draw(5,3) node(d) {$\pi_1 M$};
\draw[->](a)--(b) node [midway,above](TextNode){$\varphi$};
\draw[->](c)--(d) node [midway,above](TextNode){$\tau_{*}$};
\draw[->](c)--(a) node [midway, left] (TextNode) {$\pi_*$};
\draw[->](d)--(b) node [midway, right] (TextNode) {$\pi_*$};
\end{tikzpicture}
\end{equation*}
which tells us that, for any $\rho: \OF \to SO(3)$, the pull back representation $\pi^*\rho$ has the property that $\tau^* (\pi^*\rho) = u\cdot (\pi^*\rho)\cdot u^{-1}$ with $u = \rho(\mu_i)$ of order two.
\end{proof}

\begin{example}\label{E:rp3}
Let $\ML \subset S^3$ be the Hopf link then $M = \RP^3$ and the orbifold exact sequence \eqref{E:orb} takes the form
\[
\begin{tikzpicture}
\draw (1,1) node (a) {$1$};
\draw (3,1) node (b) {$\Z/2$};
\draw (6,1) node (c) {$\Z/2\,\oplus\,\Z/2$};
\draw (9,1) node (d) {$\Z/2$};
\draw (11,1) node (e) {$1$};
\draw[->](a)--(b);
\draw[->](b)--(c) node [midway,above](TextNode){$\pi_{*}$};
\draw[->](c)--(d) node [midway,above](TextNode){$j$};
\draw[->](d)--(e);
\end{tikzpicture}
\]
with the two copies of $\Z/2$ in the middle group generated by the meridians $\mu_1$ and $\mu_2$. Define $\rho: \Z/2\,\oplus\,\Z/2 \to SO(3)$ on the generators by $\rho (\mu_1) = \Ad i$ and $\rho (\mu_2) = \Ad j$; up to conjugation, this is the only representation $\Z/2 \to SO(3)$ with $w_2(\rho) \neq 0$. The pull back representation $\pi^*\rho:  \Z/2 \to SO(3)$ sends the generator to $\Ad i \cdot \Ad j = \Ad k$. Since $\tau^* (\pi^*\rho) = \pi^*\rho$, the identity $\tau^* (\pi^*\rho) =u\cdot (\pi^*\rho) \cdot u^{-1}$ holds for multiple choices of $u$ including the second order $u$ of the form $u = \Ad q$, where $q$ is any unit quaternion such that $- qk = kq$.
\end{example}

Given a double branched cover $\pi: M \to \Sigma$ with branch set $\ML$ and the covering translation $\tau: M \to M$, define
\[
{\MCR}_{\omega}(M; SO(3))\, =\, \{\,\beta: \pi_1 M \to SO(3)\;|\;w_2(\beta)\neq 0\,\}/\Ad SO(3).
\]
Since $w_2 (\tau^* \beta) = w_2 (\beta) \in H^2 (M;\Z/2) = \Z/2$, the pull back of representations via $\tau$ gives rise to a well defined involution 
\begin{equation}\label{E:w-inv}
\tau^*: {\MCR}_{\omega}(M; SO(3)) \longrightarrow {\MCR}_{\omega}(M; SO(3)). 
\end{equation}
Its fixed point set $\Fix (\tau^*)$ consists of the conjugacy classes of representations $\beta: \pi_1 M \to SO(3)$ such that $w_2 (\beta) \neq 0$ and there exists an element $u \in SO(3)$ having the property that $\tau^*\beta = u\cdot \beta\cdot u^{-1}$. Consider the sub-variety
\begin{equation}\label{E:uuu}
\MCR^{\tau}_w (M;SO(3))\;\subset\; \Fix (\tau^*)
\end{equation}
defined by the condition that the conjugating element $u$ can be chosen to be of order two. This sub-variety is well defined because all elements of order two in $SO(3)$ are conjugate to each other. The following proposition is the main result of this section.

\begin{prop}\label{pullback-homeo}
The homomorphism $\pi_*: \pi_1 M \to \OF$ of the orbifold exact sequence \eqref{E:orb} induces via the pull back a homeomorphism
\[
\pi^*: \MCR_{\omega} (\Sigma,\ML; SO(3)) \longrightarrow \MCR^{\tau}_{\omega}(M; SO(3)).
\] 
\end{prop}

\begin{proof}
Orbifold representations $\OF \to SO(3)$ with non-trivial $w_2$ pull back to representations $\pi_1 M \to SO(3)$ with non-trivial $w_2$, see Section \ref{S:bundles}. In addition, these pull back representations are equivariant in the sense of Lemma \ref{L:orb}. Therefore, the map $\pi^*: \MCR_{\omega} (\Sigma,\ML; SO(3)) \longrightarrow \MCR^{\tau}_{\omega}(M; SO(3))$ is well defined. To finish the proof, we will construct an inverse of $\pi^*$. Given $\beta: \pi_1 M \to SO(3)$ whose conjugacy class belongs to $\MCR^{\tau}_{\omega}(M; SO(3))$, there exists an element $u \in SO(3)$ of order two such that $\tau^* \beta = u\,\cdot\,\beta\,\cdot\,u^{-1}$. The pair $(\beta,u)$ then defines an $SO(3)$ representation of $\OF = \pi_1 M\,\rtimes\,\Z/2$ by the formula $\rho (x,t^{\ell}) = \beta(x)\,\cdot\,u^{\ell}$, where $x \in \pi_1 M$ and $t$ is the generator of $\Z/2$. 
\end{proof}


\subsection{Equivariant index}\label{S:eq-ind}
All orbifolds we encounter in this paper are obtained by taking the quotient of a smooth manifold by an orientation preserving involution. The orbifold elliptic theory on such global quotient orbifolds is equivalent to the equivariant elliptic theory on their branched covers; see for instance \cite{wang}. In particular, the orbifold index of the ASD operator \eqref{E:DA} can be computed as an equivariant index as explained below.

Let $X$ be a smooth oriented Riemannian 4-manifold without boundary, which may or may not be compact. If $X$ is not compact, we assume that its only non-compactness comes from a product end $(0,\infty) \times Y$ equipped with a product metric. Let $\tau: X \to X$ be a smooth orientation preserving isometry of order two with non-empty fixed point set $F$ making $X$ into a double branched cover over $X'$ with branch set $F'$. Let $P \to X$ be an $SO(3)$ bundle to which $\tau$ lifts so that its action on the fibers over the fixed point set of $\tau$ has order two. This lift will be denoted by $\tilde\tau: P \to P$. The quotient of $P$ by the involution $\tilde\tau$ is naturally an orbifold $SO(3)$ bundle $P' \to X'$, and any equivariant connection $A$ in $P$ gives rise to an orbifold connection $A'$ in $P'$. The ASD operator
\[
\DA(X) = -d^*_A\,\oplus\,d^+_A: \;\Omega^1 (X, \ad P) \to (\Omega^0\,\oplus\,\Omega^2_+) (X,\ad P)
\]
associated with $A$ is equivariant in that the diagram
\begin{equation*}
\begin{tikzpicture} [scale =1]
\draw(1,3) node(a) {$\Omega^1 (X,\ad P)$};
\draw(7,3) node(b) {$(\Omega^0\oplus\Omega^2_+) (X,\ad P)$};
\draw(1,1) node(c) {$\Omega^1 (X,\ad P)$ };
\draw(7,1) node(d) {$(\Omega^0\oplus\Omega^2_+) (X,\ad P)$};
\draw[->](a)--(c) node [midway,left](TextNode){$\tilde\tau^*$};
\draw[->](b)--(d) node [midway,right](TextNode){$\tilde\tau^*$};
\draw[->](a)--(b) node [midway, above] (TextNode) {$\DA\,(X)$};
\draw[->](c)--(d) node [midway, above] (TextNode) {$\DA\,(X)$};
\end{tikzpicture}
\end{equation*}
commutes, giving rise to the orbifold operator $\DAP (X'): \Omega^1 (X',\ad P') \to (\Omega^0 \oplus\Omega^2_+) (X',\ad P')$. From this we immediately conclude that 
\begin{equation}\label{E:orb=eq}
\ind \DAP (X')\; =\; \ind \D^{\tau}_A\,(X),
\end{equation}
where $\D^{\tau}_A\, (X)$ is the operator $\DA(X)$ restricted to the $(+1)$--eigen\-spaces of the involution $\tilde\tau^*$. If $X$ is closed, the operators in \eqref{E:orb=eq} are automatically Fredholm. If $X$ has a product end, we ensure Fredholmness by completing with respect to the weighted Sobolev norms
\[
\| \varphi \|_{L^2_{k,\delta}(X)}\; =\; \| h\cdot \varphi \|_{L^2_k(X)}
\]

\smallskip\noindent
where $h: X \to \R$ is a smooth function which is $\tau$--invariant and which, over the end, takes the form $h(t,y) = e^{\delta t}$ for a sufficiently small positive $\delta$. We choose to work with these particular norms to match the global boundary conditions of Atiyah, Patodi, and Singer \cite{APS:I}.

In particular, if $\rho$ and $\sigma$ are non-degenerate critical points of the orbifold Chern--Simons functional on $\Sigma$, they pull back to the flat connections $\pi^*\rho$ and $\pi^*\sigma$ on the double branched cover $M \to \Sigma$. The formula \eqref{E:rel-gr} for the relative Floer grading can then be written as
\[
\gr(\rho,\sigma)\; =\; \ind \D^{\tau}_A\, (\pistar\rho,\,\pistar\sigma)\;\;\text{(mod 4)},
\]

\smallskip\noindent
where $A$ is an equivariant connection on $\R\times Y$ which limits at the negative and positive end to $\pistar\rho$ and $\pistar\sigma$, respectively. The index in the above formula can be understood as the $L^2_{\delta}$ index for any sufficiently small $\delta \ge 0$ because the operator $\D^{\tau}_A\, (\pistar\rho,\,\pistar\sigma)$ is Fredholm in the usual $L^2$ Sobolev completion.


\subsection{Index formulas}
Let us continue with the setup of the previous subsection. One can easily see that
\medskip
\[
\ind \D^{\tau}_A (X)\, =\, \frac 1 2\,\ind \DA (X) + \frac 1 2\, \ind\, (\tau,\DA)(X),
\]
where
\[
\ind\, (\tau,\DA)(X) = \tr\,(\tilde\tau^* | \ker \DA(X)) - \tr\,(\tilde\tau^* | \coker \DA(X)).
\]

\smallskip\noindent
We will use this observation together with the standard index theorems to obtain explicit formulas for the index of operators in question.

\begin{prop}\label{P:index-formula}
Let $X$ be a closed manifold then
\[
\ind \D^{\tau}_A\,(X)\, =\, - p_1 (P) - \frac 3 4\,(\sigma(X) + \chi (X)) + \frac 1 4\,(\,\chi(F) + F\cdot F).
\]
\end{prop}

\smallskip

\begin{proof}
The index of $\DA (X)$ can be expressed topologically using the Atiyah--Singer index theorem \cite{AS}. Since the operator $\DA$ has the same symbol as the positive chiral Dirac operator twisted by $S^+\otimes\,(\ad P)_{\C}$, see \cite{AHS}, we obtain
\begin{multline}\notag
\ind \DA(X) = \int_X \widehat A\,(X) \ch (S^+) \ch (\ad P)_{\C} \\ 
=\, \int_X -2\,p_1 (A) - \frac 1 2\;p_1\,(TX) - \frac 3 2\,e\,(TX) \\
=\,-2\,p_1 (P) - \frac 3 2\,(\sigma (X) + \chi(X)),
\end{multline} 
using the Hirzebruch signature theorem in the last line. A similar expression for $\ind\,(\tau,\DA) (X)$ is obtained using the $G$--index theorem of Atiyah--Singer \cite{AS}. For the twisted Dirac operator in question, an explicit calculation in Shanahan \cite[Section 19]{Sh} leads us to the formula
\begin{alignat*}{1}
\ind\,(\tau,\DA)(X) = -\frac 1 2\,\int_F (e (TF) + e(NF)) & \ch_g (\ad P)_{\C} \\ &= \frac 1 2\,(\,\chi(F) + F\cdot F).
\end{alignat*}
Here, $TF$ and $NF$ are the tangent and the normal bundle of the fixed point set $F \subset X$, and the zero-order term in $\ch_g (\ad P)_{\C}$ equals $-1$ because this is  the trace of the second order $SO(3)$ operator acting on the fiber. Adding these formulas together, we obtain the desired formula.
\end{proof}

\begin{remark}
Our formula matches the index formulas for $\ind \DAP (X')$ of Kronheimer--Mrowka \cite[Lemma 2.11]{KM:khovanov} and Wang \cite[Theorem 18]{wang},
\bigskip
\[
\ind \DAP (X')\, =\, - p_1 (P) - \frac 3 2\,(\sigma(X') + \chi (X')) + \chi(F') + \frac 1 2\,F'\cdot F',
\]

\bigskip\noindent
after taking into account that $F'\cdot F' = 2\,(F\cdot F)$, $\chi (F) = \chi (F')$, $2 \chi (X') = \chi (X) + \chi (F)$, and $2\sigma (X') = \sigma (X) + F\cdot F$, see for instance Viro \cite{V}.
\end{remark}

Next, let $X$ be a manifold with a product end $(0,\infty)\times Y$, where $Y$ need not be connected, and work with the $L^2_{\delta}$ norms for sufficiently small $\delta > 0$. In a temporal gauge over the  end, the operator $\DA(X)$ takes the form $\DA(X) = \p/\p t + K_{A(t)}$.

\begin{prop}\label{P:eq-ind}
Let $X$ be a manifold with product end as above, and $A$ an equivariant connection which limits to a flat connection $\beta$ over the end. Then 
\begin{multline}\notag
\ind \D^{\tau}_A\,(X) = \frac 1 2\,\int_X \widehat A\,(X) \ch (S^+) \ch (\ad P)_{\C} + \frac 1 4\,(\chi(F) + F\cdot F) \\ - \frac 1 4\,(h_{\beta} - \eta_{\beta}(0)) - \frac 1 4\,(h^{\tau}_{\beta} - \eta^{\tau}_{\beta}(0)).
\end{multline}
\end{prop}

\medskip
The notations here are as follows: $h_{\beta}$ is the dimension of $H^0 (Y; \ad\beta)\,\oplus\, H^1 (Y; \ad\beta)$, $h^{\tau}_{\beta}$ is the trace of the map induced by $\tilde\tau^*$ on $H^0 (Y; \ad\beta)\,\oplus\,H^1 (Y; \ad\beta)$, $\eta_{\beta}(0)$ is the Atiyah--Patodi--Singer spectral asymmetry of $K_{\beta}$, and $\eta^{\tau}_{\beta}(0)$ its equivariant version defined as follows. For any eigenvalue $\lambda$ of the operator $K_{\beta}$, the $\lambda$--eigenspace $W^{\beta}_{\lambda}$ is acted upon by $\tilde\tau^*$ with trace $\tr (\tilde\tau^* |\, W^{\beta}_{\lambda})$. The infinite series 
\[
\eta^{\tau}_{\beta} (s) = \sum_{\lambda \neq 0}\;\sign\lambda \cdot \tr (\tilde\tau^* |\, W^{\beta}_{\lambda})\;|\lambda|^{-s}
\]
converges for $\Re(s)$ large enough and has a meromorphic continuation to the entire complex $s$--plane with no pole at $s = 0$; see Donnelly \cite{don}. This makes $\eta^{\tau}_{\beta} (0)$ a well-defined real number.

\begin{proof}[Proof of Proposition \ref{P:eq-ind}]
The index $\ind \DA (X)$ can be computed using the index theorem of Atiyah, Patodi and Singer \cite{APS:I},
\medskip
\[
\ind \DA\,(X) = \int_X \widehat A\,(X) \ch (S^+) \ch (\ad P)_{\C} - \frac 1 2\,(h_{\beta} - \eta_{\beta}(0))(Y),
\]

\medskip\noindent
and $\ind (\tau, \DA)\,(X)$ using its equivariant counterpart, the $G$--index theorem of Donnelly \cite{don}, 
\smallskip
\[
\ind\,(\tau,\DA)\,(X) = \frac 1 2 \int_F\; (e (TF) + e(NF)) - \frac 1 2\,(h^{\tau}_{\beta} - \eta^{\tau}_{\beta}(0))(Y).
\]

\medskip\noindent
The desired formula now follows because, according to the Gauss--Bonnet theorem,
\medskip
\[
\int_F\; e(TF) = \chi (F)\quad\text{and}\quad \int_F\; e(NF) = F \cdot F.
\]
\end{proof}

\begin{example}
Let $P \to Y$ be a trivial $SO(3)$ bundle with an involution $\tilde\tau$ acting as a second order operator on the fibers. Application of Proposition \ref{P:eq-ind} to the product connection $A$ on the manifold $X = \R \times Y$ results in the formula $\ind \D^{\tau}_{\theta}\,(X) = - 1$, which corresponds to the fact that the $(+1)$-eigenspace of the involution $\tilde\tau^*: H^0 (X;\ad \theta) \to H^0 (X;\ad\theta)$ is one-dimensional. 
\end{example}


\subsection{Proof of Lemma \ref{L:deg}}\label{S:deg}
Since both $\rho$ and $\sigma$ are irreducible and non-degenerate, $\gr(\chi_1 \cdot \rho, \sigma) = \gr(\chi_1 \cdot \rho, \rho) + \gr(\rho, \sigma)$. Therefore, we only need to compute $\gr(\chi_1 \cdot \rho, \rho)\pmod 4$.

Let $g \in \MG$ be a gauge transformation matching $\rho$ and $\chi_1 \cdot \rho$. The mapping torus of $g$ is an orbifold bundle $P_0$ over $S^1 \times \Sigma$ and 
\[
\gr(\chi_1 \cdot \rho, \rho)\; = \; \ind \DA (S^1 \times \Sigma) \pmod 4,
\]
for any choice of orbifold connection $A$ in $P_0$. Let $M$ be the double branched cover of $\Sigma$ with branch set $\ML$ then the index in the above formula, treated as an equivariant index on $S^1 \times M$, equals $-p_1 (Q_0)$ by the formula of Proposition \ref{P:index-formula} applied to the pull-back bundle $Q_0 = \pi^*P_0$. This reduces the above formula to
\[
\gr(\chi_1 \cdot \rho, \rho)\; =\, -\, p_1 (Q_0) \pmod 4.
\]

To compute the Pontryagin number $p_1 (Q_0)$ we observe that the bundle $Q_0$ on $S^1 \times M$ can be obtained from the bundle $Q = \pi^* P$ on $M$ as the mapping torus of a gauge transformation matching $\pi^*\rho$ with $\pi^* (\chi_1\cdot \rho) = \eta \cdot \pi^*\rho$, where $\eta = \pi^* \chi_1 \in H^1 (M; \Z/2)$. According to Braam--Donaldson \cite[Part II]{BD}, Propositions 1.9 and 1.13, 
\[
p_1 (Q_0)\; =\; 2\cdot (\,\eta\cup w_2 (Q)\, +\, \eta\cup\eta\cup\eta\,)\,[M]\pmod 4.
\]
We already know that $w_2 (Q)$ is a generator of $H^2(M;\Z/2) = \Z/2$, see Proposition \ref{P:Q}. It follows from Lemma \ref{L:branched} that the class $\eta$ is a generator of $H^1 (M;\Z/2) = \Z/2$. The desired formula now follows from the calculation of the cohomology ring $H^* (M;\Z/2)$ in Proposition \ref{P:H2}.


\section{Knot homology: the generators}\label{S:gen}
We will now use the equivariant theory of Section \ref{S:equiv} to better understand the chain complex $\IC(k)$ which computes the singular instanton knot homology $\I (k) = I_*(S^3,\k)$ of Kronheimer and Mrowka \cite{KM:khovanov}. In this section, we will describe the conjugacy classes of projective $SU(2)$ representations on the exterior of $\k$ with non-trivial $[c]$ and separate them into the orbits of the canonical $\Z/2\,\oplus\,\Z/2$ action. The next two sections will be dedicated to computing Floer gradings.


\subsection{Projective representations}
Given a knot $k \subset S^3$, denote by $K = S^3 - \Int N(k)$ its exterior and by $\KN = S^3 - \Int N(\k)$ the exterior of the two-component link $\k = k\,\cup\,\ell$ obtained by connect summing $k$ with the Hopf link. The Wirtinger presentation
\[
\pi_1 K = \langle a_1, a_2, \cdots , a_n \; | \; r_1, \ldots , r_m \rangle
\]
with meridians $a_i$ and relators $r_j$ gives rise to the Wirtinger presentation
\[
\pi_1 \KN = \langle a_1, a_2, \cdots , a_n, b \;| \; r_1, \ldots , r_m,\, [a_1, b] =  1\rangle,
\]
where $b$ stands for the meridian of the component $\ell$. Since the link $\k$ is not split, it follows from Lemma \ref{L:hopf} that $H^2 (\pi_1 \KN;\Z/2) = H^2 (\KN;\Z/2) = \Z/2$. The generator of the latter group evaluates non-trivially on both boundary components of $\KN$, which makes it Poincar\'e dual to any arc connecting these two boundary components. It follows from Proposition \ref{P:proj} that the projective representations with non-trivial $[c]$ which we are interested in are precisely the homomorphisms $\rho: F \to SU(2)$ of the free group $F$ generated by the meridians $a_1,\ldots,a_n, b$ such that
\[
\rho(r_1) = \ldots = \rho(r_n) = 1\quad\text{and}\quad \rho([a_1,b]) = -1.
\]
Representations $\rho$ are uniquely determined by the $SU(2)$ matrices $A_i = \rho(a_i)$ and $B = \rho(b)$ subject to the above relations, and the space $\PR_c (\KN,SU(2))$ consists of all such tuples $(A_1,\ldots, A_n; B)$ up to conjugation.

Observe that the relation $A_1 B = - B A_1$ implies that, up to conjugation, $A_1 = i$ and $B = j$. Since the Wirtinger relations  $r_1 = 1, \cdots, r_m = 1$ are of the form $a_i a_j a_i^{-1} = a_k$, all the matrices $A_i$ must have zero trace. In particular, the matrices $A_1 = \ldots = A_n = i$ and $B = j$ satisfy all of the relations, thereby giving rise to the special projective representation $\alpha = (i,i,\ldots, i;j)$. On the other hand, if we assume that not all $A_i$ commute with each other, we have an entire circle of projective representations,
\begin{equation}\label{E:tuple}
(i, e^{i \varphi} A_2 \, e^{-i \varphi}, \cdots, e^{i \varphi} A_n \, e^{-i \varphi}; j).
\end{equation}
It is parameterized by $e^{2i\varphi} \in S^1$ due to the fact that the center of $SU(2)$ is the stabilizer of the adjoint action of $SU(2)$ on itself. Note that two tuples like \eqref{E:tuple} are conjugate if and only if they are equal to each other. One can easily see that the formula $\psi (A_1,\ldots,A_n; B) = (A_1,\ldots,A_n)$ defines a surjective map 
\begin{equation}\label{E:pi}
\psi: \PR_c (\KN, SU(2)) \to \mathcal R_0(K,SU(2)), 
\end{equation}
where $\mathcal R_0 (K,SU(2))$ is the space of the conjugacy classes of traceless representations $\rho_0: \pi_1 K \to SU(2)$. If $\rho_0$ is irreducible, the fiber $C(\rho_0) = \psi^{-1} ([\rho_0])$ is a circle of the form \eqref{E:tuple}. The special projective representation $\alpha$ is a fiber of \eqref{E:pi} in its own right over the unique (up to conjugation) reducible traceless representation $\pi_1 K \to H_1 (K) \to SU(2)$ sending all the meridians to the same traceless matrix $i$. Therefore, assuming that $\mathcal R_0(K,SU(2))$ is non-degenerate, the space $\PR_c (\KN,SU(2))$ consists of an isolated point and finitely many circles, one for each conjugacy class of irreducible representations in $\mathcal R_0 (K,SU(2))$. The same result holds in general after perturbation. 


\subsection{The action of $H^1(\KN; \Z/2)$}
The group $H^1(\KN; \Z/2) = \Z/2\,\oplus\,\Z/2$ generated by the duals $\chi_k$ and $\chi_{\ell}$ of the meridians of the link $\k = k\,\cup\,\ell$ acts on the space of projective representations $\PR_c (\KN, SU(2))$ as explained in Section \ref{S:proj}. In terms of the tuples \eqref{E:tuple}, the generators $\chi_k$ and $\chi_{\ell}$ send $(i,e^{i\varphi} A_2 e^{-i\varphi},\ldots,e^{i\varphi} A_n e^{-i\varphi}; j)$ to
\begin{gather}
(-i,-e^{i\varphi} A_2 e^{-i\varphi},\ldots,-e^{i\varphi} A_n e^{-i\varphi}; j)\;\,\text{and} \notag \\
(i,e^{i\varphi} A_2 e^{-i\varphi},\ldots,e^{i\varphi} A_n e^{-i\varphi}; -j),\; \notag
\end{gather}
respectively. The isolated point $\alpha = (i,i,\ldots,i;j)$ is a fixed point of this action since $(-i,-i,\ldots,-i;j) = j\cdot (i,i,\ldots,i;j)\cdot j^{-1}$ and $(i,i,\ldots,i;-j) = i\cdot (i,i,\ldots,i;j)\cdot i^{-1}$.

To describe the action of $\chi_{\ell}$ on the circle $C(\rho_0)$ for an irreducible $\rho_0$ conjugate $(i,e^{i\varphi} A_2 e^{-i\varphi},\ldots,e^{i\varphi} A_n e^{-i\varphi}; -j)$ by $i$ to obtain
\[
(i,e^{i(\varphi+\pi/2)} A_2 e^{-i(\varphi+\pi/2)},\ldots,e^{i(\varphi+\pi/2)} A_n e^{-i(\varphi+\pi/2)}; j).
\]
Since the circle $C(\rho_0)$ is parameterized by $e^{2i\varphi}$, we conclude that the involution $\chi_{\ell}$ acts on $C(\rho_0)$ via the antipodal map.

The action of $\chi_k$ on the circle $C(\rho_0)$ for an irreducible $\rho_0$ will depend on whether $\rho_0$ is a binary dihedral representation or not. Recall that a representation $\rho_0: \pi_1 K \to SU(2)$ is called \emph{binary dihedral} if it factors through a copy of the binary dihedral subgroup $S^1\,\cup\,j\cdot S^1 \subset SU(2)$, where $S^1$ stands for the circle of unit complex numbers. Equivalently, $\rho_0$ is binary dihedral if its adjoint representation $\Ad(\rho_0): \pi_1 K \to SO(3)$ is \emph{dihedral} in that it factors through a copy of $O(2)$ embedded into $SO(3)$ via the map $A \to (A,\det A)$. 

One can show that a representation $\rho_0$ is binary dihedral if and only if $\chi \cdot \rho_0$ is conjugate to $\rho_0$, where $\chi: \pi_1 K \to \Z/2$ is the generator of $H^1 (K;\Z/2) = \Z/2$. Note that $\chi$ defines an involution on $\mathcal R_0 (K,SU(2))$ which makes the following diagram commute

\begin{equation*}
\begin{tikzpicture} [scale =1]
\draw(1,1) node(a) {$\PR_c (\KN,SU(2))$};
\draw(5,1) node(b) {$\mathcal R_0 (K,SU(2))$.};
\draw(1,3) node(c) {$\PR_c (\KN,SU(2))$};
\draw(5,3) node(d) {$\mathcal R_0 (K,SU(2))$};
\draw[->](c)--(a) node [midway,left](TextNode){$\chi_k$};
\draw[->](d)--(b) node [midway,right](TextNode){$\chi$};
\draw[->](c)--(d) node [midway, above] (TextNode) {$\pi$};
\draw[->](a)--(b) node [midway, above] (TextNode) {$\pi$};
\end{tikzpicture}
\end {equation*}

The action of $\chi_k$ can now be described as follows. If an irreducible $\rho_0: \pi_1 K \to SU(2)$ is not binary dihedral, the involution $\chi_k$ takes the circle $C(\rho_0)$ to the circle $C(\chi\cdot \rho_0)$. Since $\chi \cdot \rho_0$ is not conjugate to $\rho_0$, these two circles are disjoint from each other, and $\chi_k$ permutes them. If an irreducible $\rho_0: \pi_1 K \to SU(2)$ is binary dihedral, there exists $u \in SU(2)$ such that $uiu^{-1} = -i$ and $u A_i u^{-1} = -A_i$ for $i = 2,\ldots,n$. The irreducibility of $\rho_0$ also implies that $u^2 = -1$ so after conjugation we may assume that $u = k$. Now conjugate $\chi_k\cdot (i,e^{i\varphi} A_2 e^{-i\varphi},\ldots,e^{i\varphi} A_n e^{-i\varphi};j) = (-i,-e^{i\varphi} A_2 e^{-i\varphi},\ldots,-e^{i\varphi} A_n e^{-i\varphi}; j)$ by $j$ to obtain
\begin{align*}
(i, j(- e^{i\varphi} A_2 e^{-i \varphi})& j^{-1}, \cdots, j(- e^{i\varphi} A_n e^{-i \varphi})j^{-1}; j) \\
= (i, - e^{-i\varphi} j & A_2 j^{-1}  e^{i \varphi}, \cdots, - e^{-i\varphi} j A_n j^{-1}  e^{i \varphi}; j) \\
= (i, - & (ie^{-i\varphi})\,k A_2 k^{-1}  (i^{-1}e^{i \varphi}), \cdots, - (ie^{-i\varphi})\,k A_n k^{-1}  (i^{-1}e^{i \varphi}); j) \\
&=\, (i, e^{i(\pi/2 - \varphi)} A_2 e^{-i(\pi/2 - \varphi)}, \cdots, e^{i(\pi/2 - \varphi)} A_n e^{-i(\pi/2 - \varphi)}; j).
\end{align*}
Therefore, $\chi_k$ acts on $C(\rho_0)$ by sending $e^{2i\varphi}$ to $-e^{-2i\varphi}$, which is an involution on the complex unit circle with two fixed points, $i$ and $-i$.

Finally, observe that the quotient of $\mathcal R_0 (K,SU(2))$ by the involution $\chi$ is precisely the space $\mathcal R_0 (K,SO(3))$ of the conjugacy classes of representations $\Ad \rho_0: \pi_1 K \to SO(3)$. Since $H^2 (K;\Z/2) = 0$, every $SO(3)$ representations lifts to an $SU(2)$ representations, hence $\mathcal R_0 (K,SO(3))$ can also be described as the space of the conjugacy classes of representations $\pi_1 K \to SO(3)$ sending the meridians to $SO(3)$ matrices of trace $-1$. Compose \eqref{E:pi} with the projection $\mathcal R_0 (K,SU(2)) \to \mathcal R_0 (K,SO(3))$ to obtain a surjective map $\psi: \PR_c (\KN, SU(2)) \to \mathcal R_0 (K,SO(3))$. The above discussion can now be summarized as follows.

\begin{prop}\label{P:action}
The group $H^1(\KN; \Z/2) = \Z/2\,\oplus\,\Z/2$ acts on the space $\PR_c(\KN, SU(2))$ preserving the fibers of the map $\psi: \PR_c (\KN, SU(2)) \to \mathcal R_0 (K,SO(3))$. Furthermore,
\begin{enumerate}
\item[(a)] for the unique reducible in $\mathcal R_0 (K,SO(3))$, the fiber of $\psi$ consists of just one point, which is the conjugacy class of the special projective representation $\alpha$. This point is fixed by both $\chi_k$ and $\chi_{\ell}$;
\item[(b)] for any dihedral representation in $\mathcal R_0 (K,SO(3))$, the fiber of $\psi$  is a circle. The involution $\chi_k$ is a reflection of this circle with two fixed points, while $\chi_{\ell}$ is the antipodal map;
\item[(c)] otherwise, the fiber of $\psi$ consists of two circles. The involution $\chi_k$ permutes these circles, while $\chi_{\ell}$ acts as the antipodal map on both.
\end{enumerate}
\end{prop}

It should be noted that perturbing the Chern--Simons functional \eqref{E:cs} may easily break the $\Z/2\,\oplus\,\Z/2$ symmetry. Finding a perturbation which preserves this symmetry runs as usual into the equivariant transversality problem, which we do not try to address here. It should be noted, however, that such a problem was successfully solved in \cite{RSI} in a similar setting.


\subsection{Double branched covers}
Next, we would like to describe the space $\PR_c(\KN,SU(2))$ using the equivariant theory of Section \ref{S:equiv}. We could proceed as in that section, by passing to the double branched cover $M \to  S^3$ with branch set the link $\k$ and working with the equivariant representations $\pi_1 M \to SO(3)$. However, in the special case at hand, one can observe that $M$ is simply the connected sum  $Y\,\#\,\RP^3$, where $Y$ is the double branched cover of $S^3$ with branch set the knot $k$, hence the same information about $\PR_c (\KN,SU(2))$ can be extracted more easily by working directly with $Y$ and using Proposition \ref{P:action}. The only missing step in this program is a description of $\MCR_0 (K,SO(3))$ in terms of equivariant representations $\pi_1 Y \to SO(3)$, which we will take up next.

Every representation $\rho: \pi_1 K \to SO(3)$ gives rise to a representation of the orbifold fundamental group $\pi_1^V (S^3,k) = \pi_1 K/\langle \mu^2 = 1\rangle$, where we choose $\mu = a_1$ to be our meridian. The latter group can be included into the split orbifold exact sequence 
\[
\begin{tikzpicture}
\draw (1,1) node (a) {$1$};
\draw (3,1) node (b) {$\pi_1 Y$};
\draw (6,1) node (c) {$\pi_1^{V}(S^3, k)$};
\draw (9,1) node (d) {$\Z/2$};
\draw (11,1) node (e) {$1.$};
\draw[->](a)--(b);
\draw[->](b)--(c) node [midway,above](TextNode){$\pi_{*}$};
\draw[->](c)--(d) node [midway,above](TextNode){$j$};
\draw[->](d)--(e);
\end{tikzpicture}
\]

\begin{prop}\label{P:eq-knots}
Let $Y$ be the double branched cover of $S^3$ with branch set a knot $k$ and let $\tau: Y \to Y$ be the covering translation. The pull back of representations via the map $\pi_*$ in the orbifold exact sequence establishes a homeomorphism
\[
\pi^*: \MCR_0 (K,SO(3)) \longrightarrow \MCR^{\tau} (Y,SO(3)),
\]
where $\MCR^{\tau} (Y)$ is the fixed point set of the involution $\tau^*: \MCR (Y,SO(3)) \to \MCR (Y,SO(3))$. The unique reducible representation in $\MCR_0 (K,SO(3))$ pulls back to the trivial representation of $\pi_1 Y$, and the dihedral representations in $\MCR_0 (K,SO(3))$ are the ones and only ones that pull back to reducible representations of $\pi_1 Y$.
\end{prop}

\begin{proof}
A slight modification of the argument of Proposition \ref{pullback-homeo}, see also \cite[Proposition 3.3]{CS}, establishes a homeomorphism between $\MCR_0 (K,SO(3))$ and the subspace of $\MCR^{\tau} (Y,SO(3))$ consisting of the conjugacy classes of representations $\beta: \pi_1 Y \to SO(3)$ such that $\tau^*\beta = u\cdot \beta \cdot u^{-1}$ for some $u \in SO(3)$ of order two. The proof of the first statement of the proposition will be complete after we show that this subspace in fact comprises the entire space $\MCR^{\tau} (Y,SO(3))$. 

If $\beta: \pi_1 Y \to SO(3)$ is reducible, it factors through a representation $H_1 (Y) \to SO(2)$. According to Proposition \ref{P:kawauchi}, the involution $\tau_*$ acts on $H_1 (Y)$ as multiplication by $-1$. Therefore, $\tau^*\beta = \beta^{-1}$, and the latter representation can obviously be conjugated to $\beta$ by an element $u \in SO(3)$ of order two. If $\beta: \pi_1 Y \to SO(3)$ is irreducible, the condition $\beta \in \Fix (\tau^*)$ implies that there exists a unique $u \in SO(3)$ such that $\tau^*\beta = u\cdot \beta \cdot u^{-1}$ and $u^2 = 1$. Suppose that $u = 1$ then $\tau^* \beta = \beta$, which implies that $\beta$ is the pull back of a representation of $\pi_1^V (S^3,k)$ which sends the meridian $\mu$ to the identity matrix and hence factors through $\pi_1 S^3 = 1$. This contradicts the irreducibility of $\beta$.

To prove the second statement of the proposition, observe that the homomorphism $j$ in the above orbifold exact sequence sending $\mu$ to the generator of $\Z/2$ is in fact the abelianization homomorphism. This implies that the unique reducible representation in $\MCR_0 (K,SO(3))$ pulls back to the trivial representation of $\pi_1 Y$. Since $\pi_1 Y$ is the commutator subgroup of $\pi_1^V (S^3,k)$, any dihedral representation $\rho: \pi_1^V (S^3,k) \to O(2)$ must map $\pi_1 Y$ to the commutator subgroup of $O(2)$, which happens to be $SO(2)$. This ensures that the pull back of 
$\rho$ is reducible. Conversely, if the pull back of $\rho$ is reducible, its image is contained in a copy of $SO(2)$, and the image of $\rho$ itself in its 2-prime extension. The latter group is of course just a copy of $O(2) \subset SO(3)$. 
\end{proof}

\begin{remark}\label{R:conj}
For future use note that, for any projective representation $\rho: \pi_1 \KN \to SU(2)$ in $C(\rho_0)$ described by a tuple \eqref{E:tuple}, the adjoint representation $\Ad\rho: \pi_1 \KN \to SO(3)$ pulls back to an $SO(3)$ representation of $\pi_1 (Y\,\#\,\RP^3) = \pi_1 Y * \Z/2$ of the form $\beta * \gamma: \; \pi_1 Y * \Z/2 \to SO(3)$, where $\beta = \pi^* \Ad \rho_0$ and $\gamma: \Z/2 \to SO(3)$ sends the generator of $\Z/2$ to $\Ad i\cdot \Ad j = \Ad k$. The representation $\beta * \gamma$ is equivariant, $\tau^*(\beta*\gamma) = u\cdot (\beta* \gamma)\cdot u^{-1}$, with the conjugating element $u = \Ad \rho_0 (a_1) = \Ad i$.
\end{remark}


\section{Knot homology: grading of the special generator}\label{S:special}
Given a knot $k \subset S^3$, we will continue using the notations $K$ for its exterior and $\KN$ for the exterior of the two-component link $\k = k\,\cup\,\ell$ obtained by connect summing $k$ with the Hopf link $h$. The special projective representation $\alpha: \pi_1 \KN \to SU(2)$, which sends all the meridians of $k$ to $i$ and the meridian of $\ell$ to $j$, is a generator in the chain complex $\IC(k)$. In this section, we compute its absolute Floer grading.

\begin{theorem}\label{T:gr}
For any knot $k$ in $S^3$, we have $\gr (\alpha) = \sign k \pmod 4$.
\end{theorem}

Before we go on to prove this theorem, recall the definition of $\gr(\alpha)\pmod 4$. Let $(W',S)$ be a cobordism of pairs $(S^3,u)$ and $(S^3,k)$, where $u$ is an unknot in $S^3$. The manifold $W'$ is required to be oriented but the surface $S$ is not. Construct a new cobordism $(W',S')$ of the pairs $(S^3,h)$ and $(S^3,\k)$ by letting $S'$ be the disjoint union of $S$ with the normal circle bundle along a path in $S$ connecting the two boundary components (the surface $S'$ is called $S^{\natural}$ in \cite[Section 4.3]{KM:khovanov}). According to \cite[Proposition 4.4]{KM:khovanov}, the generator $\alpha$ has grading
\medskip
\begin{equation}\label{E:gr}
\gr(\alpha)\, =\, -\ind \DAP (\alpha_u,\alpha) - \dfrac{3}{2}\left(\chi(W') + \sigma(W')\right) - \chi(S')\pmod 4,
\end{equation}

\medskip\noindent
where $\alpha_u$ stands for the special generator in the Floer chain complex of $u$, and we use the fact that $\chi(S) = \chi(S')$. The operator $\DAP (\alpha_u,\alpha)$ refers to the ASD operator on the non-compact manifold obtained from $W'$ by attaching cylindrical ends to the two boundary components; this manifolds is again called $W'$. The connection $A'$ can be any connection on $W'$ which is singular along the surface $S'$ and which limits to flat connections with the holonomies $\alpha_u$ and $\alpha$ on the two ends. The index of $\DAP(\alpha_u,\alpha)$ is understood as the $L^2_{\delta}$ index for a small positive $\delta$.


\subsection{Constructing the cobordism}
Our calculation of the Floer index $\gr(\alpha)$ will use a specific cobordism $(W',S')$ constructed as follows. 

Let  $\Sigma$ be the double branched cover of $S^3$ with branch set the knot $k$. Choose a Seifert surface $F'$ of $k$ and push its interior slightly into the ball $D^4$ so that the resulting surface, which we still call $F'$, is transversal to $\partial D^4 = S^3$. Let $V$ be the double branched cover of $D^4$ with branch set the surface $F'$. Then $V$ is a smooth simply connected spin 4-manifold with boundary $\Sigma$, which admits a handle decomposition with only 0-- and 2--handles, see Akbulut--Kirby  \cite[page 113]{AK}.

Next, choose a point in the interior of the surface $F' \subset D^4$. Excising a small open 4-ball containing that point from $(D^4, F')$ results in a manifold $W'_1$ diffeomorphic to $I \times S^3$ together with the surface $F'_1 = F' - \Int(D^2)$ properly embedded into it, thereby providing a cobordism $(W'_1,F'_1)$ from an unknot to the knot $k$. The double branched cover $W_1 \to W'_1$ with branch set $F'_1$ is a cobordism from $S^3$ to $\Sigma$. The manifold $W_1$ is simply connected because it can be obtained from the simply connected manifold $V$ by excising an open 4-ball.

Similarly, consider the manifold $W'_2 = I\, \times\, S^3$ and surface $F'_2 = I\, \times\, h \subset W'_2$ providing a product cobordism from the Hopf link $h$ to itself. The double branched cover $W_2 \to W'_2$ with branch set $F'_2$ is then a cobordism $W_2 = I \times \RP^3$ from $\RP^3$ to itself.

As the final step of the construction, consider a path $\gamma'_1$ in the surface $F'_1$ connecting its two boundary components. Similarly, consider a path $\gamma'_2$ of the form $I \times \{p\}$ in the surface $F'_2 = I \times H$. Remove tubular neighborhoods of these two paths and glue the resulting manifolds and surfaces together using an orientation reversing diffeomorphism $1 \times h: I \times S^2 \to I \times S^2$. The resulting pair $(W',S')$ is the desired cobordism of the pairs $(S^3,h)$ and $(S^3,\k)$. One can easily see that 
\begin{equation}\label{E:chi}
\chi(W') = \sigma(W') = 0\quad \text{and} \quad \chi(S') = \chi (F') - 1.
\end{equation}

Note that the double branched cover $W \to W'$ with branch set $S'$ is a cobordism from $\RP^3$ to $\Sigma\,\#\,\RP^3$ which can be obtained from the cobordisms $W_1$ and $W_2$ by taking a connected sum along the paths $\gamma_1 \subset W_1$ and $\gamma_2 \subset W_2$ lifting, respectively, the paths $\gamma'_1$ and $\gamma'_2$. To be precise, 
\begin{equation}\label{E:W}
W = \WC_1\,\cup\, \WC_2, 
\end{equation}
where $\WC_1$ and $\WC_2$ are obtained from $W_1$ and $W_2$ by removing tubular neighborhoods of $\gamma_1$ and $\gamma_2$. The identification in \eqref{E:W} is done along a copy of $I \times S^2$. In particular, we see that $\pi_1 W = \Z/2$.


\subsection{$L^2$--index}\label{S:L2}
We will rely on Ruberman \cite{R:forms} and Taubes \cite{T:periodic} in our index calculations. 

Let $\pi: W \to W'$ be the double branched cover with branch set $S'$ constructed in the previous section, and $\tau: W \to W$ the covering translation. Let us consider a representation $\rho: \pi_1^V (W',S') \to SO(3)$ sending the two sets of meridians of $S'$ to $\Ad i$ and $\Ad j$. Then the representation $\pistar\rho: \pi_1 W \to SO(3)$ sends the generator of $\pi_1 W$ to $\Ad k$ and it is equivariant in that $\tau^*(\pistar\rho) = u\cdot \pistar\rho\cdot u^{-1}$ with $u = \Ad i$, compare with Example \ref{E:rp3}. The representation $\rho$ restricts to $\alpha_u$ and $\alpha$ over the two ends of $W'$, therefore,  $\pistar\rho$ makes $W$ into a flat cobordism between $\gamma: \pi_1 (\RP^3) \to SO(3)$ and $\theta * \gamma: \pi_1 \Sigma\, *\, \pi_1 (\RP^3) \to SO(3)$, where $\gamma$ is the representation of Example \ref{E:rp3}.

Let $A$ and $A'$ be flat connections on $W$ and $W'$ whose holonomies are, respectively, $\pistar\rho$ and $\rho$. We will use $A'$ as the twisting connection of the operator $\DAP(\alpha_u,\alpha)$. Instead of computing the index of this operator, we will compute the equivariant index $\ind \D^{\tau}_A (\gamma,\,\theta*\gamma)$ of its pull--back to $W$. The latter index equals minus the equivariant index of the elliptic complex
\smallskip
\begin{equation}\label{E:def-complex}
\begin{CD}
\Omega^{0}(W, \ad P) @> - d_A >> \Omega^{1}(W, \ad P) @> d^+_A > \empty > \Omega_+^{2}(W, \ad P) 
\end{CD}
\end{equation}

\smallskip\noindent
The equivariance here is understood with respect to a lift of $\tau: W \to W$ to the bundle $\ad P$ which has second order on the fibers over the fixed point set. The connection $A$ is equivariant with respect to this lift, hence it splits the coefficient bundle $\ad P$ into a sum of three real line bundles corresponding to $\Ad k = \diag (-1,-1,1)$. Accordingly, the complex \eqref{E:def-complex} splits into a sum of three elliptic complexes, one with the trivial real coefficients and two with the twisted coefficients. Application of \cite[Proposition 4.1]{R:forms} to the former complex and of \cite[Corollary 4.2]{R:forms} to the latter two reduces the index problem to computing the singular cohomology 
\smallskip
\[
H^k (W; \ad \pistar\rho) = H^k (W; \R)\,\oplus\,H^k (W; \R_-)\,\oplus\,H^k (W; \R_-),\;\; k = 0, 1, 2,
\]

\smallskip\noindent
where $\R_-$ stands for the real line coefficients on which $\Z/2$ acts as multiplication by $-1$, and their equivariant versions. 

The zeroth equivariant cohomology of complex \eqref{E:def-complex} vanishes because $H^0 (W;\R_-) = 0$ and the lift of $\tau$ acts as minus identity on the remaining group $H^0 (W;\R) = \R$. This vanishing result could also be derived directly from the irreducibility of the singular connection $A'$. The next two subsections are dedicated to computing the first and second cohomology of \eqref{E:def-complex}.


\subsection{Trivial coefficients}
Our computation will be based on the Mayer--Vietoris exact sequence applied twice, first to compute cohomology of $\WC_1$ and $\WC_2$, and then to compute cohomology of $W = \WC_1\, \cup\, \WC_2$. 

The cohomology groups of $\WC_1$ and $W_1 = \WC_1\, \cup\, (I \times D^3)$ are related by the following Mayer--Vietoris exact sequence

\begin{equation*}
\begin{tikzpicture}
\draw (0,2) node (p) {$0$};
\draw (2,2) node (q) {$H^{1}(W_1;\R)$};
\draw (6,2) node (r) {$H^{1}(\WC_1;\R)$};
\draw (10,2) node (s) {$0$};
\draw(12,2) node (t) {};
\draw[->](p)--(q);
\draw[->](q)--(r);
\draw[->](r)--(s);
\draw[->](s)--(t) ;

\draw (0,1) node (p) {};
\draw (2,1) node (q) {$H^{2}(W_1;\R)$};
\draw (6,1) node (r) {$H^{2}(\WC_1;\R)$};
\draw (10,1) node (s) {$H^{2}(I \times S^2;\R)$};
\draw(12,1) node (t) {};
\draw[->](p)--(q);
\draw[->](q)--(r); 
\draw[->](r)--(s); 
\draw[->](s)--(t) ;

\draw (0,0) node (u) {};
\draw (2,0) node (v) {$H^{3}(W_1;\R)$};
\draw (6,0) node (w) {$H^{3}(\WC_1;\R)$};
\draw (10,0) node (x) {$0$,};
\draw(12,0) node (y) {};
\draw[->](u)--(v)  node [midway,above](TextNode){$\delta$};
\draw[->](v)--(w);
\draw[->](w)--(x);

\end{tikzpicture}
\end{equation*}

\medskip\noindent
Since $W_1$ and therefore $\WC_1$ are simply connected, both $H^1 (W_1;\R)$ and $H^1(\WC_1;\R)$ vanish. Applying the Poincar\'e--Lefschetz duality to the manifold $W_1$ and using the long exact sequence of the pair $(W_1,\p W_1)$ we obtain 
\[
H^3 (W_1;\R) = H_1 (W_1,\p W_1;\R) = \widetilde H_0 (\p W_1;\R) = \R.
\]
Similarly, viewing $\WC_1$ as a manifold whose boundary is a connected sum of the two boundary components of $W_1$, we obtain 
\[
H^3 (\WC_1;\R) = H_1 (\WC_1,\p \WC_1;\R) = \widetilde H_0 (\p \WC_1;\R) = 0.
\]
Therefore, the connecting homomorphism $\delta$ in the above exact sequence must be an isomorphism, which leads to the isomorphisms
\[
H^2(\WC_1;\R) = H^2 (W_1;\R) = H^2 (V;\R).
\]
A similar long exact sequence relates the cohomology of $\WC_2$ and $W_2 = \WC_2\, \cup\, (I \times D^3)$, implying that 
\[
H^2 (\WC_2;\R) = H^2(W_2;\R) = H^2 (\RP^3;\R) = 0.
\]
Since $\pi_1 W_2 = \pi_1 \WC_2 = \Z/2$, both $H^1 (W_2;\R)$ and $H^1(\WC_2;\R)$ vanish. The Mayer--Vietoris exact sequence of the splitting $W = \WC_1\,\cup\,\WC_2$,

\begin{equation*}
\begin{tikzpicture}
\draw (0,2) node (p) {$0$};
\draw (2,2) node (q) {$H^{1}(W;\R)$};
\draw (6,2) node (r) {$H^{1}(\WC_1;\R)\,\oplus\, H^{1} (\WC_2;\R)$};
\draw (10,2) node (s) {$0$};
\draw(12,2) node (t) {};
\draw[->](p)--(q);
\draw[->](q)--(r);
\draw[->](r)--(s);
\draw[->](s)--(t) ;

\draw (0,1) node (p) {};
\draw (2,1) node (q) {$H^{2}(W;\R)$};
\draw (6,1) node (r) {$H^{2}(\WC_1;\R)\,\oplus\, H^{2} (\WC_2;\R)$};
\draw (10,1) node (s) {$H^{2}(I \times S^2;\R)$};
\draw(12,1) node (t) {};
\draw[->](p)--(q);
\draw[->](q)--(r); 
\draw[->](r)--(s); 
\draw[->](s)--(t) ;

\draw (0,0) node (u) {};
\draw (2,0) node (v) {$H^{3}(W;\R)$};
\draw (6,0) node (w) {$H^{3}(\WC_1;\R)\,\oplus\, H^{3} (\WC_2;\R)$};
\draw (10,0) node (x) {$0$};
\draw(12,0) node (y) {};
\draw[->](u)--(v); 
\draw[->](v)--(w);
\draw[->](w)--(x);

\end{tikzpicture}
\end{equation*}

\medskip\noindent
together with the isomorphisms $H^3 (W;\R) = H_1 (W,\p W;\R) = \widetilde H_0 (\p W;\R) = \R$ and $\pi_1 W = \Z/2$, implies that
\[
H^1 (W;\R) = 0 \quad \text{and} \quad H^2 (W;\R) = H^2 (V;\R).
\]


\subsection{Twisted coefficients}
We will now do a similar calculation using the Mayer--Vietoris sequence of $W =\WC_1\,\cup\,\WC_2$ with twisted coefficients. Since $\WC_1$ is simply connected, the twisted coefficients $\R_-$ pull back to the trivial $\R$--coefficients over $\WC_1$ and the cohomology calculations from the previous section are unchanged. A direct calculation using homotopy equivalences $W_2\, \simeq\, \RP^3$ and $\WC_2 \simeq \RP^2$ shows that 
\[
H^1(\WC_2; \R_-) = 0\quad \text{and} \quad H^2 (\WC_2; \R_-) = \R.
\] 
The latter isomorphism is induced by the inclusion $I \times S^2 \to \WC_2$, which can be easily seen from the Mayer--Vietoris exact sequence of $W_2 = \WC_2\,\cup\,(I \times D^3)$. Now, consider the Mayer--Vietoris exact sequence of the splitting $W = \WC_1\,\cup\,\WC_2$ with twisted $\R$--coefficients,

 \begin{equation*}
\begin{tikzpicture}
\draw (0,2) node (p) {$0$};
\draw (2,2) node (q) {$H^{1}(W;\R_-)$};
\draw (6,2) node (r) {$H^{1}(\WC_1;\R)\,\oplus\, H^{1} (\WC_2;\R_-)$};
\draw (10,2) node (s) {$0$};
\draw(12,2) node (t) {};
\draw[->](p)--(q);
\draw[->](q)--(r);
\draw[->](r)--(s);
\draw[->](s)--(t) ;

\draw (0,1) node (p) {};
\draw (2,1) node (q) {$H^{2}(W;\R_-)$};
\draw (6,1) node (r) {$H^{2}(\WC_1;\R)\,\oplus\, H^{2} (\WC_2;\R_-)$};
\draw (10,1) node (s) {$H^{2}(I \times S^2;\R)$};
\draw(12,1) node (t) {};
\draw[->](p)--(q);
\draw[->](q)--(r); 
\draw[->](r)--(s); 
\draw[->](s)--(t) ;

\draw (0,0) node (u) {};
\draw (2,0) node (v) {$H^{3}(W;\R_-)$};
\draw (6,0) node (w) {$H^{3}(\WC_1;\R)\,\oplus\, H^{3} (\WC_2;\R_-)$};
\draw (10,0) node (x) {$0$.};
\draw(12,0) node (y) {};
\draw[->](u)--(v);  
\draw[->](v)--(w);
\draw[->](w)--(x);

\end{tikzpicture}
\end{equation*}

\medskip\noindent
Keeping in mind that the map $H^2 (\WC_1;\R) \to H^2 (I\times S^2;\R)$ in this sequence is zero and the map $H^2 (\WC_2;\R_-) \to H^2 (I\times S^2;\R)$ is an isomorphism $\R \to \R$, we conclude that
\[
H^1 (W;\R_-) = 0 \quad \text{and} \quad H^2 (W;\R_-) = H^2 (V;\R).
\]


\subsection{Equivariant cohomology}\label{S:eq}
Combining results of the previous two sections we obtain $H^1 (W; \ad P) = 0$ and $H^2 (W; \ad P) = H^2 (V; \R^3)$. The action of $\tau$ is compatible with these isomorphisms, from which we immediately conclude that 
\[
H^1_{\tau} (W; \ad P) = 0
\]
and $H^2_{\tau} (W; \ad P)$ is the fixed point set of the map $H^2 (V; \R^3) \to H^2 (V; \R^3)$ obtained by twisting $\tau^*: H^2 (V; \R) \to H^2 (V; \R)$ by the action on the coefficients $\R^3 \to \R^3$. The involution $\tau^*$ is minus the identity, which follows from the usual transfer argument applied to the covering $V \to D^4$, while the action on the coefficients is given by an  $SO(3)$ operator of second order. Such an operator must have a single eigenvalue $1$ and a double eigenvalue $-1$, which leads us to the conclusion that $\rk H^2_{\tau}\, (W; \ad P) = 2\cdot b_2 (V)$. Similarly, 
\[
\rk H^2_{\tau,+} (W; \ad P) = 2\cdot b_2^+ (V). 
\]


\subsection{Proof of Theorem \ref{T:gr}}
It follows from the discussion in Section \ref{S:L2} and the calculation in Section \ref{S:eq} that
\[
\ind \DAP (\alpha_u,\alpha)\, =\, \rk H^1_{\tau}\, (W;\ad P) - \rk H^2_{+,\tau} (W;\ad P) = -2\cdot b_2^+ (V).
\]
Taking into account \eqref{E:gr} and \eqref{E:chi}, we obtain the formula
\[
\gr (\alpha)\, =\, 2\cdot b_2^+ (V)\, - \, \chi (F') + 1 \pmod 4.
\]
To simplify it, let us compute $\chi (V)$ in two different ways: $\chi (V) = 1 + b_2^+ (V) + b_2^- (V)$ by definition, and $\chi (V) = 2 \chi (D^4) - \chi (F') = 2 - \chi (F')$ using the fact that $V$ is a double branched cover of $D^4$ with branch set $F'$. Combining these formulas with the knot signature formula of Viro \cite{V}, we obtain the desired result (remember that $\sign k$ is always even)\,:
\[
\gr (\alpha)\, =\, - \sign V\, =\, - \sign k\; = \; \sign k \pmod 4.
\]

\medskip


\section{Knot homology: gradings of other generators}\label{S:other}
Proposition \ref{P:action} identified the critical points of the Chern--Simons functional with the fibers of the map $\psi: \PR_c (\KN, SU(2)) \to \mathcal R_0 (K,SO(3))$. Assuming that the space $\mathcal R_0 (K,SO(3))$ is non-degenerate, all of these fibers (with the exception of the special generator $\alpha$) are Morse--Bott circles. In this section, we will compute their Floer gradings using the equivariant index theory of Section \ref{S:eq-ind}. The actual generators of the chain complex $\IC(k)$ are then obtained by perturbing each Morse--Bott circle of index $\mu$ into two points of indices $\mu$ and $\mu + 1$ as in \cite{HHK}. Our index calculation will depend on whether an irreducible trace-free representation $\rho_0: \pi_1 K \to SO(3)$ giving rise to the Morse--Bott circle $C(\rho_0)$ is dihedral or not. The two cases will be considered separately starting with the case when $\rho_0$ is dihedral. If $\mathcal R_0 (K,SO(3))$ fails to be non-degenerate, similar results hold after additional perturbations.


\subsection{Dihedral representations}\label{S:dih-index}
Let $\rho_0: \pi_1 K \to SO(3)$ be an irreducible trace-free dihedral representation. The pull back via $\pi: M \to \Sigma$ identifies the Morse--Bott circle $C(\rho_0)$ with the circle of the conjugacy classes of equivariant representations of the form $\beta * \gamma: \pi_1 Y\,*\,\Z/2 \to SO(3)$, where $\beta$ is a non-trivial reducible representation of $\pi_1 Y$ and $\gamma$ is the representation of $\Z/2$ sending the generator to $\Ad k$. These representations are equivariant in that $\tau^* (\beta * \gamma) = u\cdot (\beta * \gamma)\cdot u^{-1}$ with $u = \Ad i$, see Remark \ref{R:conj}.

We wish to compute the equivariant index $\ind \D^{\tau}_A\,(\beta*\gamma,\theta*\gamma)$, where $A$ is any equivariant connection on the cylinder $\R\, \times\, (Y\,\#\,\RP^3)$ limiting to the flat connections $\beta*\gamma$ and $\theta *\gamma$ over the negative and positive ends, respectively. The Morse--Bott index of the circle corresponding to $\beta * \gamma$ will then equal 
\begin{equation}\label{E:red}
\mu\; =\; \ind \D^{\tau}_A\, (\beta*\gamma,\theta*\gamma)\, +\, \sign k\pmod 4.
\end{equation}

\smallskip

\begin{prop}\label{P:red}
Let $\beta: \pi_1 Y \to SO(3)$ be a non-trivial equivariant reducible representation then, for any equivariant connection $B$ on the cylinder $\R \times Y$ limiting to the flat connections $\beta$ and $\theta$ over the negative and positive ends, 
\[
\ind \D^{\tau}_A\, (\beta*\gamma,\theta*\gamma)\; =\; \ind \D^{\tau}_B\, (\beta,\theta)\pmod 4.
\]
\end{prop}

\begin{proof}
To compute the index on the left-hand side of this formula, we will apply the formula of Proposition \ref{P:eq-ind} to the manifold $X = \R \times (Y\,\#\,\RP^3)$ with two product ends. Since the metric on $X$ is a product metric, the terms $p_1\,(TX)$ and $e (TX)$ in the integrand
\smallskip
\[
\widehat A\,(X) \ch (S^+) \ch (\ad P)_{\C}\;=\;-2\,p_1 (A) - \frac 1 2\;p_1\,(TX) - \frac 3 2\,e\,(TX)
\]

\medskip\noindent
will vanish, as will the topological terms $\chi(F)$ and $F\cdot F$, leading to the formula 
\smallskip
\begin{alignat*}{1}
\ind \D^{\tau}_A\,(\beta*\gamma,\theta*\gamma) = -\int_X\; p_1 (A) & - \frac 1 4\,(h_{\theta*\gamma} - \rho_{\theta*\gamma}) - \frac 1 4\,(h_{\beta*\gamma} + \rho_{\beta*\gamma}) \\
&- \frac 1 4\,(h^{\tau}_{\theta*\gamma} - \rho^{\tau}_{\theta*\gamma}) - \frac 1 4\,(h^{\tau}_{\beta*\gamma} + \rho^{\tau}_{\beta*\gamma})
\end{alignat*}

\medskip\noindent
where $\rho_{\beta*\gamma} = \eta_{\beta*\gamma}(0) - \eta_{\theta}(0)$ and $\rho^{\tau}_{\beta*\gamma} = \eta^{\tau}_{\beta*\gamma}(0) - \eta^{\tau}_{\theta}(0)$ are $\rho$--invariants of the manifold $Y\,\#\,\RP^3$. 

The connection $A$ in this formula is any equivariant connection limiting to the flat connections $\beta*\gamma$ and $\theta*\gamma$ at the two ends of $X$, hence we are free to choose $A$ to equal $\gamma$ over $\R \times (\RP^3 - D^3)$ and to be trivial in the gluing region. This evaluates the integral term in the above formula as follows 
\medskip
\[
\int_X\; p_1 (A) = \int_{\,\R\times Y}\; p_1 (A).
\]

\medskip
To evaluate the $\rho$--invariants, build a cobordism $W$ from the disjoint union $Y\cup\,\RP^3$ to the connected sum $Y\,\#\,\RP^3$ by attaching a 1-handle to $[0,1] \times (Y \cup\,\RP^3)$. The flat connection $\beta*\gamma$ extends to $W$ making it into a flat cobordism from $(Y,\beta)\,\cup\,(\RP^3,\gamma)$ to $(Y \#\,\RP^3,\beta*\gamma)$. It then follows from \cite[Theorem 2.4]{APS:II} that
\[
\rho_{\beta*\gamma} - \rho_\beta - \rho_\gamma\, =\, \sign_{\beta*\gamma} (W) - 3\sign (W),
\]
where $\rho_{\beta}$ and $\rho_{\gamma}$ are $\rho$--invariants of the manifolds $Y$ and $\RP^3$, respectively. One can easily see from the description of $W$ that both signature terms in the above formula vanish implying that $\rho_{\beta*\gamma} = \rho_{\beta} + \rho_{\gamma}$. Since the involution $\tau$ extends to $W$, a similar argument using the index theorem of Donnelly \cite{don} instead of \cite[Theorem 2.4]{APS:II} shows that $\rho^{\tau}_{\beta*\gamma} = \rho^{\tau}_{\beta} + \rho^{\tau}_{\gamma}$. Similar formulas also hold with $\theta*\gamma$ in place of $\beta*\gamma$.

Plugging all of this back into the above index formula and keeping in mind that $\rho_{\theta} = \rho^{\tau}_{\theta} = 0$, we obtain
\medskip
\begin{alignat*}{1}
\ind \D^{\tau}_A\,(\beta*\gamma,\theta*\gamma) = -\int_{\R \times Y}\; p_1 (A) & - \frac 1 4\,(h_{\beta*\gamma} + \rho_{\beta}) - \frac 1 4\,h_{\theta*\gamma} \\
&- \frac 1 4\,(h^{\tau}_{\beta*\gamma} + \rho^{\tau}_{\beta}) - \frac 1 4\,h^{\tau}_{\theta*\gamma}.\quad
\end{alignat*}

\medskip\noindent
On the other hand, one can apply the formula of Proposition \ref{P:eq-ind} to the manifold $X = \R \times Y$ to obtain 
\medskip
\begin{alignat*}{1}
\;\;\ind \D^{\tau}_A\,(\beta,\theta) = -\int_{\R\times Y}\; p_1 (A) & - \frac 1 4\,(h_{\beta} + \rho_{\beta}) - \frac 1 4\,h_{\theta} \\
&- \frac 1 4\,(h^{\tau}_{\beta} + \rho^{\tau}_{\beta}) - \frac 1 4\,h^{\tau}_{\theta}.
\end{alignat*}

\medskip\noindent
Therefore,
\medskip
\begin{alignat*}{1}
\qquad \ind \D^{\tau}_A\,(\beta*\gamma,\theta*\gamma) - \ind \D^{\tau}_A\,(\beta,\theta) = &
- \frac 1 4\,(h_{\beta*\gamma} - h_{\beta}) - \frac 1 4\,(h_{\theta*\gamma} - h_{\theta}) \\ &
- \frac 1 4\,(h^{\tau}_{\beta*\gamma} - h^{\tau}_{\beta}) - \frac 1 4\,(h^{\tau}_{\theta*\gamma} - h^{\tau}_{\theta}),
\end{alignat*}

\medskip\noindent
and the proof of the proposition reduces to a calculation with twisted cohomology. 

Since $Y$ is a rational homology sphere, $H^1 (Y; \ad\theta) = 0$. Therefore, $h_{\theta} = \dim H^0 (Y; \ad\theta) = 3$ and $h^{\tau}_{\theta} = \tr\,(\Ad u) = -1$. It follows from a calculation in Section \ref{S:special} that $H^1 (Y\,\#\,\RP^3; \ad(\theta*\gamma)) = 0$. Therefore, $h_{\theta*\gamma} = \dim H^0 (Y; \ad(\theta*\gamma)) = 1$ because $H^0 (Y; \ad(\theta*\gamma))$ is the $(+1)$--eigenspace of $\Ad (k): \so(3) \to \so(3)$. The operator $\Ad i$ acts as minus identity on the $(+1)$--eigenspace of $\Ad k$ making $h^{\tau}_{\theta*\gamma} = -1$.

The calculation with $\beta*\gamma$ will rely on the Mayer--Vietoris exact sequence of the splitting $Y\,\#\,\RP^3 = Y_0\,\cup\,\RP^3_0$ with twisted coefficients\,:
\medskip
\[
\begin{tikzpicture}
\draw (0,2) node (p) {$0$};
\draw (3,2) node (q) {$H^0 (Y\,\#\,\RP^3; \,\ad(\beta*\gamma))$};
\draw (8.5,2) node (r) {$H^0 (Y;\ad\beta)\oplus H^0 (\RP^3;\ad\gamma)$};
\draw (12,2) node (s) {};
\draw[->](p)--(q);
\draw[->](q)--(r);
\draw[->](r)--(s);

\draw (2,1) node (p) {};
\draw (4.5,1) node (q) {$H^0 (S^2;\ad\theta)$};
\draw (9,1) node (r) {$H^1 (Y\,\#\,\RP^3; \,\ad(\beta*\gamma))$};
\draw(12,1) node (s) {};
\draw[->](p)--(q);
\draw[->](q)--(r); 
\draw[->](r)--(s); 

\draw (4,0) node (u) {};
\draw (8,0) node (v) {$H^1 (Y;\ad\beta)\oplus H^1 (\RP^3;\ad\gamma)$};
\draw (11.75,0) node (w) {$0$};
\draw[->](u)--(v); 
\draw[->](v)--(w);

\end{tikzpicture}
\]

\noindent
Since $\beta$ is reducible but non-trivial, $H^0 (Y; \ad\beta) = \R$. Therefore, keeping in mind that $H^0 (S^2; \ad\theta) = \R^3$, $H^0 (\RP^3; \ad\gamma) = \R$, and $H^1 (\RP^3; \ad\gamma) = 0$, we obtain
\[
h_{\beta*\gamma} - h_{\beta}\; =\; 2\cdot \dim H^0\,(Y\,\#\,\RP^3; \ad (\beta*\gamma)).
\]
The involution $\tau$ induces involutions $\tilde\tau^*$ on each of the groups in the Mayer--Vietoris exact sequence comprising a chain map. Keeping in mind that the traces of $\tilde\tau^*$ are equal to $-1$ on both $H^0 (S^2; \ad\theta) = \R^3$ and $H^0 (\RP^3; \ad\gamma) = \R$, we obtain
\[
h^{\tau}_{\beta*\gamma} - h^{\tau}_{\beta} = 2\tr \left(\,\tilde\tau^* |\, H^0\,(Y\,\#\,\RP^3; \,\ad (\beta*\gamma))\right) - 2\tr \left(\,\tilde\tau^* |\, H^0\,(Y; \ad\beta)\right).
\]
Even though both $\beta$ and $\gamma$ are reducible, the representation $\beta*\gamma$ may be either reducible or irreducible. In the former case, $H^0\,(Y\,\#\,\RP^3; \,\ad (\beta*\gamma)) = \R$ is the $(+1)$--eigenspace of the operator $\Ad k: \so(3) \to \so(3)$ on which $\tilde\tau^*$ acts as minus identity, therefore, $h_{\beta*\gamma} - h_{\beta} = 2$ and $h^{\tau}_{\beta*\gamma} - h^{\tau}_{\beta} = 0$. In the latter case, $H^0\,(Y\,\#\,\RP^3; \,\ad (\beta*\gamma)) = 0$, therefore, $h_{\beta*\gamma} - h_{\beta} = 0$ and $h^{\tau}_{\beta*\gamma} - h^{\tau}_{\beta} = 2$. In both cases, we conclude that
\[
\ind \D^{\tau}_A\,(\beta*\gamma,\theta*\gamma)\, =\, \ind \D^{\tau}_A\,(\beta,\theta).
\]
The result now follows from the fact that $\ind \D^{\tau}_A\,(\beta,\theta) = \ind \D^{\tau}_B\,(\beta,\theta)\pmod 4$ for any choice of connections $A$ and $B$ on the cylinder $\R\,\times\, Y$ limiting to $\beta$ and $\theta$ over the negative and positive ends.
\end{proof}

\begin{remark}\label{R:irred}
The formula of Proposition \ref{P:red} holds as well for equivariant irreducible representations $\beta$, the proof requiring just minor adjustments. 
\end{remark}

Combining Proposition \ref{P:red} with formula \eqref{E:red}, we obtain the following formula for the Floer grading.

\begin{cor}\label{C:final}
Let $\beta: \pi_1 Y \to SO(3)$ be a non-trivial equivariant reducible representation then the Floer grading of the Morse--Bott circle arising from $\beta * \gamma$ is given by
\[
\mu\; =\; \ind \D^{\tau}_B (\beta,\theta)\, +\, \sign k\pmod 4,
\]
where $B$ is an arbitrary equivariant connection on the infinite cylinder $\mathbb R\,\times\,Y$ limiting to $\beta$ and $\theta$ over the negative and positive ends.
\end{cor}

The index $\ind \D^{\tau}_B (\beta,\theta)$ in the above corollary can be computed using the formula 
\begin{equation}\label{E:d+d}
\ind \D^{\tau}_B (\beta,\theta)\, =\, \frac 1 2\,\ind \D_B (\beta,\theta) + \frac 1 2\, \ind\, (\tau,\D_B)(\beta,\theta).
\end{equation}

\medskip\noindent
According to Donnelly \cite{don}, 
\begin{multline}\notag
\ind\, (\tau,\D_B)(\beta,\theta) = \frac 1 2 \int_F\; (e (TF) + e(NF)) \\ - \frac 1 2\,(h^{\tau}_{\theta} - \eta^{\tau}_{\theta}(0))(Y) - \frac 1 2\,(h^{\tau}_{\beta} + \eta^{\tau}_{\beta}(0))(Y),
\end{multline}

\medskip\noindent
where the integral term vanishes and $h^{\tau}_{\beta} = h^{\tau}_{\theta} = -1$ as in the proof of Proposition \ref{P:red}. Therefore, 
\begin{equation}\label{E:dtau}
\ind\, (\tau,\D_B)(\beta,\theta)\, =\, 1 - \frac 1 2\cdot \rho^{\tau}_{\beta} (Y).
\end{equation}
The $\rho$--invariants in this formula are difficult to compute in general but they can be shown to vanish in several special cases, like that of two-bridge knots discussed in Section \ref{S:bridge}.


\subsection{Non-dihedral representations}\label{S:non-dih}
Let $\rho_0: \pi_1 K \to SO(3)$ be an irreducible trace-free representation which is not dihedral, and assume that it is non-degenerate. Proposition \ref{P:action}(c) then tells us that the fiber $C(\rho_0)$ consists of two circles which are permuted by the involution $\chi_k$.

\begin{lemma}
The involution $\chi_k$ permuting the two circles in $C(\rho_0)$ has degree zero mod 4.
\end{lemma}

\begin{proof}
This follows as in Lemma \ref{L:deg} whose proof in Section \ref{S:deg} needs to be amended to allow for the one-dimensional critical point sets $C(\rho_0)$. This is easily accomplished by replacing $\gr (\chi_1\cdot\rho,\rho)$ with $\gr (\chi_1\cdot\rho,\rho) + 1$ in the first two displayed formulas. 
\end{proof}

Therefore, the two circles in $C(\rho_0)$ have the same Morse--Bott index $\mu$. Perturbing both of them, we obtain four generators, two of grading $\mu$ and two of grading $\mu + 1$. The calculation of the previous section leading up to the formula of Corollary \ref{C:final} can be easily amended to work in the current situation, producing the following result.

\begin{prop}\label{P:final}
Let $\beta: \pi_1 Y \to SO(3)$ be an irreducible representation then the Floer grading of the two Morse--Bott circles arising from $\beta * \gamma$ is
\[
\mu\; =\; \ind \D^{\tau}_B (\beta,\theta)\, +\, \sign k\pmod 4,
\]
where $B$ is an arbitrary equivariant connection on the infinite cylinder $\mathbb R\,\times\,Y$ limiting to $\beta$ and $\theta$ over the negative and positive ends.
\end{prop}

The index $\ind \D^{\tau}_B (\beta,\theta)$ in this proposition can be computed using the formula \eqref{E:d+d}. Since $h^{\tau}_{\beta}$ now vanishes, the formula \eqref{E:dtau} takes the form 
\smallskip
\begin{equation}\label{E:dtau2}
\ind\, (\tau,\D_B)(\beta,\theta)\, =\, \frac 1 2\,-\,\frac 1 2 \cdot \rho^{\tau}_{\beta} (Y).
\end{equation}

\smallskip

\begin{remark}
Let $\Delta (t)$ be the Alexander polynomial of a knot $k \subset S^3$ normalized so that $\Delta(t) = \Delta(t^{-1})$ and $\Delta(1) = 1$. The knots $k$ with $\Delta(-1) = 1$ are precisely the knots whose double branched covers $Y$ are integral homology spheres, and which are known to have no dihedral representations in $\mathcal R_0 (K,SO(3))$; see \cite[Theorem 10]{Klassen} or \cite[Proposition 3.4]{CS}. Therefore, all the generators in $\IC(k)$ are of the non-dihedral type studied in this section. In addition, $\sign k = 0\pmod 8$ because $1 = \Delta (-1) = \det (i\cdot Q)$, where $Q$ is the (even) quadratic form of the knot.
\end{remark}


\section{Knot homology: explicit calculations}
The equivariant techniques work particularly well for Montesinos knots, including two-bridge and pretzel knots, as we will demonstrate in this section. We begin with two-bridge knots, then discuss the Montesinos knots whose double branched covers are integral homology spheres, and then move on to the general Montesinos knots. We finish with a short section on torus knots. 


\subsection{Two-bridge knots}\label{S:bridge}
Let $p$ be an odd positive integer and $k$ a two-bridge knot of type $-p/q$ in the 3-sphere. Its double branched cover $Y$ is the lens space $L(p,q)$ oriented as the $(-p/q)$--surgery on an unknot in $S^3$. One can use Proposition \ref{P:kawauchi} to show that all representations $\beta: \pi_1 Y \to SO(3)$ are equivariant. The invariant $\rho^{\tau}_{\beta} (Y)$ of formula \eqref{E:dtau} has been shown to vanish in \cite[Proposition 27]{S:jdg}. Therefore, $\ind (\tau, \D_B) (\beta,\theta) = 1$ and formula \eqref{E:d+d} reduces to
\[
\ind \D^{\tau}_B (\beta,\theta)\; =\; \frac 1 2\; (\ind \D_B (\beta,\theta) + 1)\pmod 4.
\]

\smallskip\noindent
Let $\beta: \pi_1 Y \to SO(3)$ be a representation sending the canonical generator of $\pi_1 Y$ to the adjoint of $\exp\,(2\pi i\ell/p)$. The quantity $\ind \D_B (\beta,\theta) + 1$ mod 8 was shown by Sasahira \cite[Corollary 4.3]{sasahira} (see also Austin \cite{austin}) to equal
\[
2N_1 (k_1,k_2)\, +\, N_2 (k_1,k_2) \pmod 8,
\]
where the integers $0 < k_1 < p$ and $0 < k_2 < p$ are uniquely determined by the equations $k_1 = \ell \pmod p$, $k_2 = -r\ell\pmod p$ and $qr = 1\pmod p$, and 
\[
N_1(k_1,k_2) = \#\,\{\,(i,j) \in \Z^2\;|\;i+qj = 0\;(\text{mod}\; p),\; |i| < k_1,\; |j| < k_2\},
\]
\begin{multline}\notag
\;\;\; N_2(k_1,k_2) = \#\,\{\,(i,j) \in \Z^2\;|\;i+qj = 0\;(\text{mod}\; p),\; \\ |i| = k_1,\; |j| < k_2,\;\text{or}\;
|i| < k_1,\; |j| = k_2\}.\;\;
\end{multline}

\begin{example} The figure-eight knot $k$ is the two-bridge knot of type $-5/3$. Its double branched cover is the lens space $L(5,3)$ whose fundamental group has no irreducible representations and has two non-trivial reducible representations, up to conjugacy. For these two representations, $\ell$ equals 1 and 2 and, by Sasahira's formula, $\ind \D_B (\beta,\theta) + 1$ equals 2 and 4 mod 8. Since $\sign k = 0$, the corresponding Morse--Bott circles have indices $\mu = 1$ and 2 mod 4 by Corollary\eqref{C:final}. After perturbation, they contribute the generators of Floer indices 1, 2 and 2, 3 mod 4, respectively. The ranks of the chain groups $\IC(k)$ are then equal to $(1,0,0,0) + (0,1,1,0) + (0,0,1,1) = (1,1,2,1)$. This equals the Khovanov homology of (the mirror image of) $k$ hence we conclude from the Kronheimer--Mrowka spectral sequence that the ranks of $\I(k)$ also equal $(1,1,2,1)$.
\end{example}


\subsection{Special Montesinos knots}\label{S:mont1}
Let $p$, $q$, and $r$ be pairwise relatively prime positive integers, and view the Brieskorn homology sphere $\Sigma(p,q,r)$ as the link of singularity at zero of the complex polynomial $x^p + y^q + z^r$. The involution induced by the complex conjugation on the link makes $\Sigma(p,q,r)$ into a double branched cover of $S^3$ with branch set a Montesinos knot which will be called $k(p,q,r)$, see for instance \cite[Section 7]{S:jdg}. 

Since $\Sigma(p,q,r)$ is an integral homology sphere, all representations $\beta: \pi_1(\Sigma (p,q,r))\to SO(3)$ apart from the trivial one are irreducible. Fintushel and Stern \cite{FS} showed that all irreducible representations $\beta$ are non-degenerate and, up to conjugation, there are $-2 \,\lambda (\Sigma (p,q,r))$ of them, where $\lambda (\Sigma(p,q,r))$ is the Casson invariant of $\Sigma(p,q,r)$. The representations $\beta$ are also equivariant, see \cite[Proposition 8]{S:jdg}, hence each conjugacy class of them contributes four generators to the chain complex $\IC (k(p,q,r))$, two of grading $\mu(\beta)$ and two of grading $\mu(\beta) + 1$. 

\begin{theorem}\label{T:br}
The ranks of the chain groups $\IC(k(p,q,r))$ are given by $(1 + b,\; b,\; b,\; b)$, where $b\, =\, -2\,\lambda(\Sigma(p,q,r))$.
\end{theorem}

\begin{proof}
Our proof will use the flat cobordism of Fintushel and Stern \cite{FS}, which is constructed as follows. The mapping torus of the Seifert fibration $\Sigma(p,q,r) \to S^2$ is an orbifold with three singular points whose neighborhoods are open cones over lens spaces. The compact manifold obtained from $W$ by excising these cones is an equivariant flat cobordism $W_0$ between $\Sigma(p,q,r)$ and the lens spaces. One can easily see that the intersection form on $H^2 (W_0; \mathbb R) = \mathbb R$ is negative definite. 

An equivariant version of \cite[Theorem 2.4]{APS:II} together with the vanishing of the $\rho^{\tau}$--invariants of lens spaces \cite[Proposition 27]{S:jdg} imply that
\smallskip
\[
\rho^{\tau}_{\beta}\, (\Sigma(p,q,r))\;=\;\sign_{\beta}\, (\tau,W_0) - \sign_{\theta}\, (\tau,W_0),
\]
where 
\[
\sign_{\beta}\,(\tau,W_0)\;=\;\tr (\tilde\tau^*\,|\,H^2_+ (W_0;\ad\beta))\, -\, \tr (\tilde\tau^*\,|\,H^2_- (W_0;\ad\beta)),
\]

\smallskip\noindent
and similarly for $\sign_{\theta}\,(\tau,W_0)$. It follows from \cite[Proposition 2.5 and Lemma 2.6]{FS} that $H^2 (W_0; \ad\beta) = 0$, therefore, $\rho^{\tau}_{\beta}\, (\Sigma(p,q,r)) = \tr\,(\Ad u) = -1$ and $\ind\, (\tau,\D_B)(\beta,\theta) = 1$ by formula \eqref{E:dtau2}. Proposition \ref{P:final} and formula \eqref{E:d+d} now imply that
\[
\mu(\beta)\; =\; \frac 1 2\,\left(\,\ind \D_B (\beta,\theta) + 1\right).
\]

\smallskip\noindent
The index $\ind \D_B (\beta,\theta)$ in this formula can be computed explicitly using either \cite{FS} or Corollary \ref{C:t-red}, however, this alone will not lead us to the closed form formula of Theorem \ref{T:br}.

Instead, we will use the 4-periodicity in the instanton Floer homology due to Fr{\o}yshov \cite[Theorem 2]{froyshov}. In the case at hand, the Floer homology of $\Sigma(p,q,r)$ equals its Floer chain complex whose generators are the conjugacy classes of irreducible representations $\beta$, hence the 4-periodicity simply means that there is a (non-canonical) free involution of degree 4 on these generators. For any pair of generators $\beta_1$ and $\beta_2$,
\smallskip
\[
\mu(\beta_2) - \mu(\beta_1)\;=\;\frac 1 2\,\left(\,\ind \D_B (\beta_2,\theta) - \ind \D_B (\beta_1,\theta)\right) \pmod 4,
\]

\smallskip\noindent
which is exactly half the relative grading of the generators $\beta_1$ and $\beta_2$ in the Floer chain complex of $\Sigma(p,q,r)$. For any involutive pair $(\beta_1,\beta_2)$, we have $\mu(\beta_2) - \mu(\beta_1) = 2\pmod 4$, therefore, each such pair contributes $(2,2,2,2)$ to the chain complex $\IC(k(p,q,r))$. The special generator $\alpha$ resides in degree zero so the result follows.
\end{proof}

\begin{example}
$\Sigma(2,3,7)$ is a double branched cover of $S^3$ whose branch set $k(2,3,7)$ is the pretzel knot $P(-2,3,7)$. Since $\lambda(\Sigma(2,3,7)) = -1$, we conclude that the ranks of the chain groups $\IC(P(-2,3,7))$ are $(3,2,2,2)$. This is consistent with the calculation in \cite[Section 5]{FKP}. \end{example}

We expect that the formula of Theorem \ref{T:br} can be proved for all Seifert fibered homology spheres $\Sigma (a_1, \ldots, a_n)$ and the corresponding Montesinos knots $k(a_1,\ldots,a_n)$ using $\tau$--equivariant perturbations of \cite{S:top3} modeled after the perturbations of Kirk and Klassen \cite{KK}. Note that the action of $H^1(K;\Z/2)$ on the conjugacy classes of projective representations is free hence it causes no equivariant transversality issues.


\subsection{General Montesinos knots}\label{S:mont2}
Let $(a_1,b_1),\ldots, (a_n,b_n)$ be pairs of integers such that, for each $i$, the integers $a_i$ and $b_i$ are relatively prime and $a_i$ is positive. Burde and Zieschang \cite[Chapter 7]{BZ} associated with these pairs a Montesinos link $K((a_1,b_1),\ldots, (a_n,b_n))$ and showed that its double branched cover is a Seifert fibered manifold $Y$ with unnormalized Seifert invariants $(a_1,b_1),\ldots,(a_n,b_n)$. In particular,
\[
\pi_1 Y = \langle\, x_1,\ldots x_n,h\;|\;h\;\text{central},\; x_i^{a_i} = h^{-b_i},\; x_1\cdots x_n = 1\,\rangle,
\]
with the covering translation $\tau: Y \to Y$ acting on the fundamental group by the rule
\[
\tau_*(h) = h^{-1},\;\; \tau_*(x_i) = x_1\cdots x_{i-1} x_i^{-1} x_{i-1}^{-1} \cdots x_1^{-1},\;i=1,\ldots, n,
\]
see Burde--Zieschang \cite[Proposition 12.30]{BZ}. The knots $k(a_1,\ldots,a_n)$ of the previous section are of the type $K((a_1,b_1),\ldots, (a_n,b_n))$; we omitted the parameters $(b_1,\ldots, b_n)$ from the notation because they can be uniquely recovered from the pairwise relatively prime $a_1, \ldots, a_n$ up to isotopy of the knot. All two-bridge and pretzel knots and links are special cases of Montesinos knots and links. In this section, we will only be interested in Montesinos knots; the case of Montesinos links of two components will be addressed in Section \ref{S:mont}.

Let $k$ be a Montesinos knot $K((a_1,b_1),\ldots, (a_n,b_n))$ and $Y$ the double branch cover of $S^3$ with branch set $k$. The manifold $Y$ need not be an integral homology sphere; in fact, one can easily see that its first homology is a finite abelian group of order 
\smallskip
\[
|H_1 (Y;\Z)|\; =\; \left(\sum_{i=1}^n\; b_i/a_i\right)\cdot a_1\cdots a_n.
\]

\smallskip\noindent
Note that this integer is always odd because $Y$ is a $\Z/2$ homology sphere. 

All reducible representations $\beta: \pi_1 Y \to SO(3)$ are equivariant because the involution $\tau_*: H_1 (Y) \to H_1 (Y)$ acts as  multiplication by $-1$, see Proposition \ref{P:kawauchi}. There are no irreducible representations for $n \le 2$. If $n = 3$, all irreducible representations are non-degenerate and equivariant, which can be shown using a minor modification of the arguments of \cite[Proposition 2.5]{FS} and \cite[Proposition 30]{S:jdg}. For $n \ge 4$, one encounters positive dimensional manifolds of representations; the action of $\tau^*$ on these manifolds was described in \cite{S:top3}, together with  equivariant perturbations making them non-degenerate. This discussion, together with Propositions \ref{P:action} and \ref{P:eq-knots}, identifies the generators of the chain complex $\IC (k)$ for all Montesinos knots in terms of representations for Seifert fibered manifolds, which are well known. An independent calculation of the generators of $\IC (k)$ for pretzel knots $k$ with $n = 3$ can be found in Zentner \cite{Z}. 

Let $W_0$ be the mapping cylinder of the Seifert fibration $Y \to S^2$ with excised open cones around its singular points. Then $W_0$ is a cobordism from a disjoint union of the lens spaces $L(a_i,-b_i)$ to $Y$.

\begin{lemma}\label{L:flat}
The cobordism $W_0$ is a flat cobordism provided $a_1\cdots a_n = \lcm(a_1,\ldots,a_n)\cdot |H_1 (Y;\Z)|$.
\end{lemma}

\begin{proof}
The fundamental group $\pi_1 W_0$ is obtained from $\pi_1 Y$ by setting the homotopy class $h \in \pi_1 Y$ of the circle fiber equal to one. Since $h$ is a central element in $\pi_1 Y$, every irreducible representation $\beta: \pi_1 Y \to SO(3)$ has the property that $\beta(h) = 1$. This property need not hold for reducible representations but it does if $h = 1$ in the first homology group $H_1 (Y)$. The algebraic condition of the lemma ensures exactly that, see Lee--Raymond \cite[page 331]{LR}.
\end{proof}

To avoid dealing with perturbations, we will assume from now on that our knot $k$ is a Montesinos knot of type $K((a_1,b_1)(a_2,b_2), (a_3,b_3))$ and that $W_0$ is a flat cobordism. We wish to calculate Floer gradings of the generators in the chain complex $\IC(k)$. Recall that every conjugacy class of non-trivial reducible representation $\beta: \pi_1 Y \to SO(3)$ gives rise to two generators of gradings $\mu(\beta)$ and $\mu(\beta) + 1$, and every conjugacy class of irreducible representation to four generators, two of grading $\mu(\beta)$ and two of grading $\mu(\beta) + 1$. The trivial representation as usual gives rise to just one generator $\alpha$ of grading $\sign k$.

\begin{lemma}
For any non-trivial representation $\beta: \pi_1 Y \to SO(3)$, we have 
\[
\mu(\beta) \;=\;\sign k\; +\; \frac 1 2\,\left(\, \ind \D_B (\beta,\theta) + 1\right)\pmod 4.
\]
\end{lemma}

\begin{proof}
This formula holds for all irreducible representations $\beta$ by the same argument as in the proof of Theorem \ref{T:br}. That argument can be easily amended for non-trivial reducible representations $\beta: \pi_1 Y \to SO(3)$ by using \eqref{E:dtau} in place of \eqref{E:dtau2}. The $\rho$--invariant in formula \eqref{E:dtau} is given by the formula 
\[
\rho^{\tau}_{\beta}\, (Y)\;=\;\sign_{\beta}\, (\tau,W_0) - \sign_{\theta}\, (\tau,W_0)
\]
with $\sign_{\theta}\,(\tau,W_0) = 1$. To compute the cohomology of $W_0$ with coefficients in $\ad\beta$, write $\ad P = \mathbb R\,\oplus\,L$, where $L$ is a line bundle with a non-trivial flat connection. Then $H^2 (W_0; L) = 0$ by the argument of \cite[Lemma 2.6]{FS} and $H^2 (W_0;\mathbb R) = \R$. Since the manifold $W_0$ is negative definite, we easily conclude that $\sign_{\beta}\, (\tau,W_0) = 1$. Therefore, $\rho^{\tau}_{\beta}\, (Y) = 0$, and the result follows.
\end{proof}

To complete the calculation of Floer gradings, we need to compute the index $\ind \D_B (\beta,\theta)$. This can be done by extending the formulas of Fintushel and Stern \cite{FS} from integral homology spheres to the more general situation at hand. We will restrict ourselves to the relatively easy case of odd $a_i$ and leave the case of even $a_i$ open because it would require passing to a double branched cover as in the proof of \cite[Theorem 3.7]{FS}. 

Given a flat cobordism $W_0$, any representation $\beta: \pi_1 (Y) \to SO(3)$ gives rise to a representation $\pi_1 (W_0) \to SO(3)$ and to representations $\beta_i: \pi_1\,(L(a_i,-b_i)) \to SO(3)$. Let us assume that $a_i$ are odd and $\beta_i \neq \theta$ for $i = 1,\ldots, m$, and that $\beta_i =\theta$ for $i = m+1,\ldots,3$. Applying the excision principle for the ASD operator twice, first to $\mathbb R\,\times\,L(a_i,-b_i)$ with $i = 1,\ldots, m$, and then to $W_0$ with the attached product ends, we obtain
\smallskip
\begin{gather}
- 3\; =\; \ind \D_B\, (\theta,\theta)\;=\;\ind \D_B\, (\theta,\beta_i)\,+\,1\,+\,\ind \D_B\, (\beta_i,\theta)\quad\text{and} \notag \\
-3\; = \;\ind \D_B\, (W_0,\theta,\theta)\;=\;
\sum_{i=1}^m\; (\,\ind \D_B\, (\theta,\beta_i) + 1\,)\hspace{0.83in} \notag \\ 
\hspace{1.9in} +\;\ind \D_B (W_0)\,+\,1\,+\,\ind \D_B\, (\beta,\theta),\notag
\end{gather}
where $\D_B (W_0)$ stands for the ASD operator on $W_0$ twisted by a flat connection $B$ whose holonomy is the representation $\pi_1 (W_0) \to SO(3)$. A similar argument with even $a_i$ does not work because representations $\beta_i$ and $\theta$ may end up living in different $SO(3)$ bundles. 

\begin{lemma}
Let $\beta: \pi_1 (Y) \to SO(3)$ be a non-trivial representation then $\ind \D_B (W_0) = -1$ if $\beta$ is reducible, and $\ind \D_B (W_0) = 0$ if $\beta$ is irreducible.
\end{lemma}

\begin{proof}
The proof of \cite[Proposition 3.3]{FS} implies the formula for irreducible $\beta$ immediately, and for reducible $\beta$ after a minor modification. To be precise, let us assume that $\beta$ is reducible. The index at hand equals $h^1 - h^0 - h^2$, where $h^0$, $h^1$, and $h^2$ are the Betti numbers of the elliptic complex
\begin{equation*}
\begin{tikzpicture}
\draw (0,0) node (a) {$0$} ;
\draw (2,0) node (b) {$\Omega^{0}(W_0, \ad P)$};
\draw (6,0) node (c) {$\Omega^{1}(W_0, \ad P)$};
\draw (10,0) node (d) {$\Omega_+^{2}(W_0, \ad P).$};
\draw[->](a)--(b);
\draw[->](b)--(c) node [midway,above](TextNode){$-d_B$};
\draw[->](c)--(d) node [midway,above](TextNode){$d^{+}_B$};
\end{tikzpicture}
\end{equation*}
Since $B$ has one-dimensional stabilizer we immediately conclude that $h^0 = 1$. To compute the remaining Betti numbers, write $\ad P = \mathbb R\,\oplus\,L$, where $L$ is a line bundle with a non-trivial flat connection. The argument of \cite[Lemma 2.6]{FS} can be used to show that the homomorphisms $H^1 (W_0; L) \to H^1 (Y;L)$ and $H^2 (W_0;L) \to H^2 (Y;L)$ induced by the inclusion $Y \to W_0$ are injective. Both $H^1 (W_0; \mathbb R)$ and $H^1 (Y;\mathbb R)$ vanish,  and the long exact sequence of the pair $(W_0,Y)$ shows that the kernel of the map $H^2 (W_0;\mathbb R) \to H^2 (Y;\mathbb R)$ is one-dimensional. Keeping in mind that the manifold $W_0$ is negative definite, we conclude as in the proof of \cite[Proposition 3.3]{FS} that $h^1 = h^2 = 0$.
\end{proof}

\begin{cor}\label{C:t-red}
Let $\beta: \pi_1 (Y) \to SO(3)$ be a non-trivial representation such that $a_i$ is odd and $\beta_i \neq \theta$ for $i = 1,\ldots,m$, and $\beta_i = \theta$ for $i = m+1,\ldots, 3$. Then
\[
\mu(\beta)\;=\;\sign\,k\,-\,1\, +\,\frac 1 2\,\sum_{i=1}^m\; (\,\ind \D_B\, (\beta_i,\theta) + 3\,) \pmod 4,
\]
where the index $\ind \D_B\, (\beta_i,\theta)$ on the infinite cylinder $\mathbb R\,\times\,L(a_i,-b_i)$  
can be computed as in Section \ref{S:bridge}.
\end{cor}

\begin{example}\label{E:233}
Let us view the pretzel knot $k = P(-2,3,3)$ as the Montesinos knot $K ((2,-1),(3,1),(3,1))$. It obviously satisfies the condition of Lemma \ref{L:flat}. Its double branched cover is a Seifert fibered manifold $Y$ whose fundamental group has presentation
\smallskip
\[
\langle\, x_1, x_2, x_3, x_4, h \; | \; h \; \text{central},\;x_1^2 = h,\; x_2^3 = h^{-1},\; x_3^3 = h^{-1},\; x_1\, x_2\, x_3 =1\, \rangle.
\]

\smallskip\noindent
This group admits one non-trivial reducible representation $\beta$ with $\beta(x_1) = 1$, $\beta (x_2) = \Ad (\exp(2\pi i/3))$ and $\beta(x_3) = \Ad(\exp(-2\pi i/3))$ contributing generators of gradings $\mu$ and $\mu + 1$ to the chain complex $\IC(k)$. To compute $\mu$, we apply the formulas of Section \ref{S:bridge} to the lens space $L(3,-1) = L(3,2)$ twice to obtain $\ind \D_B (\beta_2,\theta) = \ind \D_B (\beta_3, \theta) = 1$ mod 8. Since $\sign k = 2 \pmod 4$, it follows from Corollary \ref{C:t-red} that $\mu = 1$ mod 4 hence the contribution of $\beta$ to the chain complex is $(0,1,1,0)$. The special generator $\alpha$ contributes $(0,0,1,0)$.

The group $\pi_1 Y$ also admits one irreducible representation $\beta$ such that all of the induced representations $\beta_1: \pi_1 (L(2,1)) \to SO(3)$ and $\beta_2, \beta_3: \pi_1 (L(3,2)) \to SO(3)$ are non-trivial. Corollary \ref{C:t-red} no longer applies hence we can only conclude that the contribution of this representation to $\IC(k)$ is $(2,2,0,0)$ up to cyclic permutation. 

This information can be combined with the fact that the Kronheimer--Mrowka spectral sequence of the knot $k = P(-2,3,3)$ is trivial and that the Khovanov homology groups of $k$ have ranks $(2,1,1,1)$; see Lobb--Zentner \cite{LZ}. It then follows that the ranks of the chain groups $\IC(k)$ must be $(2,1,2,2)$, with the contribution of the irreducible being $(2,0,0,2)$, and that the boundary operator $\IC_2 (k) \to \IC_3 (k)$ must be non-trivial.
\end{example}

We expect that a similar calculation can be done for all Montesinos knots $K((a_1,b_1),\ldots,(a_n, b_n))$ satisfying the condition of Lemma \ref{L:flat} with the help of the equivariant perturbations of \cite{S:top3}. 


\subsection{Torus knots}\label{S:torus1}
Let $p$ and $q$ be positive integers which are odd and relatively prime. The double branched cover of the right handed $(p,q)$--torus knot $T_{p,q}$ is the Brieskorn homology sphere $\Sigma (2,p,q)$. According to Fintushel--Stern \cite{FS}, all irreducible $SO(3)$ representations of the fundamental group of $\Sigma (2,p,q)$ are non-degenerate and, up to conjugacy, there are $-\sign (T_{p,q})/4$ of them. All of these representations are equivariant \cite[Section 4.2]{CS} hence each of them contributes four generators to the chain complex of $\I (T_{p,q})$, two of index $\mu$ and two of 
index $\mu + 1$. Calculating $\mu$ would require equivariant index theory on the double branched cover of $T_{p,q}$ which is currently not sufficiently well developed. We know that the special generator resides in degree zero because $\sign (T_{p,q}) = 0$ mod 8, and we conjecture that the ranks of the chain groups $\IC(T_{p,q})$ are 
\[
(1 + a,\; a,\; a,\; a),\quad\text{where}\quad a = -\sign (T_{p,q})/4.
\]
This conjecture is consistent with the calculations for torus knots by Hedden, Herald and Kirk \cite{HHK}.

Let us now assume that $p$ and $q$ are relatively prime positive integers such that $p$ is odd and $q = 2r$ is even. The double branched cover $Y$, which is no longer an integral homology sphere, is the link of singularity at zero of the complex polynomial $x^2 + y^p + z^{2r} = 0$, with the covering translation given by the formula $\tau(x,y,z) = (-x,y,z)$. Neumann and Raymond \cite{NR} showed that $Y$ admits a fixed point free circle action making it into a Seifert fibration over $S^2$ with the Seifert invariants $\{(a_1,b_1),\ldots, (a_n,b_n)\} = \{(1,b_1), (p, b_2), (p, b_2), (r, b_3)\}$, where $b_1\cdot pr + 2 b_2\cdot r + b_3 \cdot p = 1$. In principle, this allows for calculation of the generators in the Floer chain complex $\IC(T_{p,q})$.

\begin{example}\label{E:34}
Let us consider the torus knot $T_{3,4}$. The manifold $Y$ has Seifert invariants $\{(1,-1), (2,1), (3,1), (3,1)\}$ which match the Seifert invariants $\{(2,-1), (3,1), (3,1)\}$ of the manifold in Example \ref{E:233}. This happens for a good reason that $P(-2,3,3)$ and $T_{3,4}$ are the same knot. The calculation of Example \ref{E:233} then tells us that the ranks of the chain groups $\IC (T_{3,4})$ are $(2,1,2,2)$, with a non-trivial boundary operator $\IC_2 (T_{3,4}) \to \IC_3 (T_{3,4})$. This is consistent with \cite{HHK}. 
\end{example}


\section{Link homology of general two-component links}\label{S:links}
This section deals with general two-component links $\ML = \ell_1 \cup \ell_2$ and not just the links $\ML = \k$ used in the definition of the knot Floer homology $\I(k)$. After computing the Euler characteristic of  $I_*(\Sigma,\ML)$, we explicitly compute the Floer chain groups for some links $\ML$ with particularly simple double branched covers.


\subsection{Euler characteristic}
Let $\ML = \ell_1\,\cup\,\ell_2$ be a two-component link in an integral homology sphere $\Sigma$. The linking number $\lk(\ell_1,\ell_2)$ is well defined up to a sign for non-oriented links $\ML$.

\begin{theorem}\label{T:euler}
The Euler characteristic of the Floer homology $I_*(\Sigma,\ML)$ of a two-component link $\ML = \ell_1\,\cup\,\ell_2$ equals $\pm\,\lk(\ell_1,\ell_2)$.
\end{theorem}

\begin{proof}
The Floer excision principle can be used as in \cite{KM:khovanov} to establish an isomorphism between $I_*(\Sigma,\ML)$ and the sutured Floer homology of $\ML$. The latter is the Floer homology of the 3-manifold $X_{\varphi}$ obtained by identifying the two boundary components of $S^3 - \Int N(\ML)$ via an orientation reversing homeomorphism $\varphi: T^2 \to T^2$. According to \cite[Lemma 2.1]{HS}, the homeomorphism $\varphi$ can be chosen so that $X_{\varphi}$ has integral homology of $S^1 \times S^2$. The result then follows from \cite[Theorem 2.3]{HS} which asserts that the Euler characteristic of the sutured Floer homology of $\ML$ equals $\pm\, \lk (\ell_1,\ell_2)$.
\end{proof}

Theorem \ref{T:euler} implies in particular that the Euler characteristic of $\I(k)$ equals $\pm 1$, which is the linking number of the two components of the link $\k$. This also follows from the fact that the critical point set of the orbifold Chern--Simons functional used to define $\I(k)$ consists of an isolated point and finitely many isolated circles, possibly after a perturbation. An absolute grading on $\I(k)$ was fixed in \cite{KM:khovanov} so that the grading of the isolated point is even; this is consistent with our Theorem \ref{T:gr} because $\sign k$ is always even. The Euler characteristic of $\I(k)$ then equals $+1$. We do not know how to fix an absolute grading on $I_*(\Sigma,\ML)$ for a general two-component link $\ML$.


\subsection{Pretzel link $P(2,-3,-6)$}\label{S:pretzel}
This is the two-component link $\ML$ whose double branched cover is the Seifert fibered manifold $M$ with unnormalized Seifert invariants $(2,1)$, $(3,-1)$, and $(6,-1)$, see for instance \cite[Section 4]{S:top2}. In particular,
\[
\pi_1 M = \langle\, x, y, z,h\;|\;h\;\text{central},\; x^2 = h^{-1},\; y^3 = h,\; z^6 = h,\; xyz = 1\,\rangle,
\]
with the covering translation $\tau: M \to M$ acting on the fundamental group by the rule
\[
\tau_*(h) = h^{-1},\;\; \tau_*(x) = x^{-1},\;\; \tau_*(y) = x y^{-1}x^{-1},\;\; \tau_*(z) = xyz^{-1}y^{-1}x^{-1},
\]
see Burde--Zieschang \cite[Proposition 12.30]{BZ}. The manifold $M$ has integral homology of $S^1 \times S^2$. In fact, it can be obtained by 0--surgery on the right-handed trefoil so that $\pi_1 M = \pi_1 K /\langle \lambda \rangle$, where $K$ is the exterior of the trefoil and $\lambda$ is its longitude. The relation $\lambda = 1$ shows up as the relation $z^6 = h$ in the above presentation of $\pi_1 M$. 

We will use this surgery presentation of $M$ to describe representations $\pi_1 M \to SO(3)$ with non-trivial $w_2 \in H^2 (M;\Z/2) = \Z/2$. According to Example \ref{E:floer}, the conjugacy classes of such representations are in one-to-two correspondence with the conjugacy classes of representations $\rho: \pi_1 K \to SU(2)$ such that $\rho (\lambda) = -1$. In the terminology of Section \ref{S:proj}, these $\rho$ are projective representations $\rho: \pi_1 M \to SU(2)$, and the group $H^1 (M;\Z/2) = \Z/2$ acts on them freely providing the claimed one-to-two correspondence. Therefore, we wish to find all the $SU(2)$ matrices $\rho(h)$, $\rho(x)$, $\rho(y)$, and $\rho(z)$ such that 
\[
\rho(x)^2 = \rho(h)^{-1},\;\; \rho(y)^3 = \rho(h),\;\; \rho(z)^6 = - \rho(h),\;\; \rho(x)\rho(y)\rho(z) = 1,
\]
and $\rho(h)$ commutes with $\rho(x)$, $\rho(y)$, and $\rho(z)$. Since $\rho$ is irreducible, we conclude as in Fintushel--Stern \cite[Section 2]{FS} that $\rho(h) = -1$ and that $\rho(x)$ is conjugate to $i$, $\rho(y)$ is conjugate to $e^{\pi i/3}$, and $\rho(z)$ is conjugate to either $e^{\pi i/3}$ or $e^{2\pi i/3}$. These give rise to two conjugacy classes of projective representations $\rho: \pi_1 M \to SU(2)$ corresponding to a single conjugacy class of representations $\Ad\rho: \pi_1 M \to SO(3)$. 

The arguments of \cite[Proposition 2.5]{FS} and \cite[Proposition 8]{S:jdg} can be easily adapted to conclude that the representation $\Ad\rho$ is non-degenerate and equivariant. It gives rise to a single $\Z/2\,\oplus\,\Z/2$ orbit of generators in $IC_*(S^3,\ML)$. Since the linking number between the components of $\ML$ is even, Lemma \ref{L:deg} tells us that the (relative) Floer indices of these four generators are $0, 0, 2, 2 \pmod 4$. The boundary operators then must vanish, and we conclude that the Floer homology groups $I_k (S^3,\ML)$ are free abelian groups of ranks $(2,0,2,0)$, up to cyclic permutation.

\begin{remark}
The same result can be obtained independently using the isomorphism between $I_*(S^3,\ML)$ and the sutured Floer homology of $\ML$ defined in \cite{KM:instanton}. The latter is the Floer homology of the manifold $X_{\varphi}$ obtained by identifying the two boundary components of $X = S^3 - \Int N(\ML)$ via an orientation reversing homeomorphism $\varphi: T^2 \to T^2$. A surgery description of $X_{\varphi}$ can be found in \cite{HS}; computing its Floer homology is then an exercise in applying the Floer exact triangle to this surgery description.
\end{remark}


\subsection{Montesinos links}\label{S:mont}
Let $(a_1,b_1),\ldots, (a_n,b_n)$ be pairs of integers such that, for each $i$, the integers $a_i$ and $b_i$ are relatively prime and $a_i$ is positive. Associated with these pairs is the Montesinos link $K((a_1,b_1),\ldots, (a_n,b_n))$ whose definition can be found for instance in \cite[Chapter 7]{BZ}. All two-bridge and pretzel links are Montesinos links; for example, the link $P(2,-3,-6)$ considered in the previous section is the Montesinos link with the parameters $(2,1)$, $(3,-1)$, and $(6,-1)$. The double branched covers $M$ of Montesinos links were described in Section \ref{S:mont2}. In this section, we will only be interested in Montesinos links whose double branched covers have integral homology of $S^1 \times S^2$, a condition that is easily checked by abelianizing $\pi_1 M$. This condition guarantees that the unique $SO(3)$ bundle $P \to M$ with non-trivial $w_2 (P) \in H^2 (M;\Z/2) = \Z/2$ does not carry any reducible connections. 

The generators of Floer chain complex of the link $K((a_1,b_1),\ldots,(a_n,b_n))$ and their gradings can be computed explicitly using the equivariant theory developed in this paper; here is a brief outline. 

Since $M$ is Seifert fibered, the representations $\pi_1 M \to SO(3)$ with non-trivial $w_2$ can be described in terms of their rotation numbers using a slight modification of the Fintushel--Stern \cite{FS} algorithm; complete details can be found in \cite{S:sbornik}. If $n = 3$, there are finitely many conjugacy classes of such representations, all of which are non-degenerate and equivariant with the conjugating element of order two. If $n \ge 4$, the same conclusion holds after using $\tau$--equivariant perturbations similar to those described in \cite{S:top3}. Note that no equivariant transversality issues are caused by the action of $H^1 (M;\Z/2)$ or $H^1 (X;\Z/2)$ because both actions are free. In what follows, we will restrict ourselves to the case of $n = 3$; however, we expect that the same results will hold for all $n$.

The relative indices of the operator $\DA$ on $\mathbb R\,\times\, M$ were computed explicitly in \cite{S:sbornik} and shown to be even. The relative Floer gradings of the generators in the Floer chain complex of the link $K((a_1,b_1),(a_2,b_2),(a_3,b_3))$ are equal to one half times those indices, by the argument of \cite[Section 5.2]{S:jdg} modified to take into account the non-triviality of the bundle $P \to M$. 

The final outcome of this calculation can be stated in terms of the Floer homology groups $I_*(M,P)$ of the unique admissible bundle $P \to M$ as follows. The groups $I_* (M,P)$ are free abelian of ranks $(n_0, n_1, n_2, n_3)$, up to cyclic permutation, with either $n_0 = n_2 = 0$ or $n_1 = n_3 = 0$. Assume for the sake of concreteness that $n_0 = n_2 = 0$ then the Floer chain groups of $K((a_1,b_1),\ldots,(a_n,b_n))$, up to cyclic permutation, have the ranks
\smallskip
\begin{equation}\label{E:nnnn}
(2n_1,\; 2n_3,\; 2n_1,\; 2n_3).
\end{equation}

\begin{example} 
The double branched cover $M$ of the Montesinos link $\ML = K((2,1),(5,-2),(10,-1))$ can be obtained by 0--surgery on the right-handed torus knot $T_{2,5}$. Applying the Floer exact triangle to this surgery, we conclude that $I_*(M,P)\,\oplus\,I_{*+4}\,(M,P) = I_*(\Sigma(2,15,11))$, where we use the mod 8 grading in both groups. Fintushel and Stern \cite{FS} showed\footnote[3]{\;We adjusted the formulas of \cite{FS} to take into account that Fintushel and Stern work with SD rather than ASD equations.} that the groups $I_k (\Sigma(2,5,11))$ are free abelian of the ranks $(0,1,0,2,0,1,0,2)$. Therefore, $n_1 = 1$, $n_3 = 2$, and the Floer chain groups of the link $\ML$ have the ranks $(2,4,2,4)$.
\end{example}

In fact, the integers $n_1$ and $n_3$ in the formula \eqref{E:nnnn} can be computed much more easily in terms of classical knot invariants without any reference to the Floer homology. They are known to satisfy the equations 
\[
-n_1 - n_3\; =\; \lambda' (M)\quad\text{and}\quad - n_1 + n_3\;=\;\bar\mu'(M),
\]
where $\lambda'(M)$ is the Casson invariant of $M$ and $\bar\mu'(M)$ its Neumann invariant \cite{N}. The former equation follows from the Casson surgery formula and the latter from \cite{S:top2}. The Casson and Neumann invariants can then be computed explicitly using the formulas
\[
\lambda' (M)\; =\;-\,1/2 \cdot \Delta''_M(1)\quad\text{and}\quad
\bar\mu' (M)\; = \; \pm\, \lk (\ell_1,\ell_2),
\]
where $\Delta_M(t)$ is the Alexander polynomial of $M$ normalized so that $\Delta_M(1) = 1$ and $\Delta(t) = \Delta(t^{-1})$, and $\lk (\ell_1,\ell_2)$ is the linking number between the components of the link $\ML$. Note that there is no need to fix the sign in the above formula because switching that sign preserves the answer \eqref{E:nnnn} up to cyclic permutation.


\section{Appendix: homology of double branched covers}\label{S:top}
This section contains a proof of Proposition \ref{P:H2} which was postponed until later in Section \ref{S:prelim}.


\subsection{Computing $H_*(M;\Z/2)$}
In this section, we will compute the groups $H_*(M;\Z/2)$ using the transfer homomorphism approach of \cite{LW}. 

The transfer homomorphisms can be defined in the following two equivalent ways, see for instance \cite[Section 3]{DK}. For each singular simplex $\sigma:\Delta \to \Sigma$, choose a lift $\tilde\sigma: \Delta \to M$ and define the chain map $\pi_{!}: C_*(\Sigma) \longrightarrow C_*(M)$ by the formula $\pi_{!}(\sigma) =\tilde\sigma + \tau \circ \tilde\sigma$. This map is obviously independent of the choice of $\tilde\sigma$, and it induces homomorphisms $\pi_{!}: H_{*}(\Sigma) \longrightarrow H_{*}(M)$ and $\pi^{!}: H^{*}(M) \longrightarrow H^{*}(\Sigma)$ in homology and cohomology with arbitrary coefficients, called transfer homomorphisms. Another way to define $\pi_{!}$ is as the map that makes the following digram commute, 
\begin{equation*}
\begin{tikzpicture} [scale =1]
\draw(1,1) node(a) {$H_*(\Sigma)$};
\draw(5,1) node(b) {$H^{*}(\Sigma)$};
\draw(1,3) node(c) {$H_{*}(M)$};
\draw(5,3) node(d) {$H^{*}(M)$};
\draw[->](a)--(c) node [midway,left](TextNode){$\pi_{!}$};
\draw[->](b)--(d) node [midway,right](TextNode){$\pi^{*}$};
\draw[->](d)--(c) node [midway, above] (TextNode) {$\PD$};
\draw[->](b)--(a) node [midway, above] (TextNode) {$\PD$};
\end{tikzpicture}
\end {equation*}
where $\PD$ stands for the Poincar\'e duality isomorphism, and similarly for $\pi^{!}$.

From now on, all chain complexes and (co)homology will be assumed to have $\Z/2$ coefficients. It is then immediate from the definition of $\pi_{!}: C_*(\Sigma) \longrightarrow C_*(M)$ that $\ker  \pi_{!} = C_{*}(\ML)$ and that we have a short exact sequence of chain complexes  
\begin{equation*}
\begin{tikzpicture}
\draw (1,1) node (a) {$0$};
\draw (3,1) node (b) {$C_{*}(\Sigma, \ML)$};
\draw (6,1) node (c) {$C_{*}(M)$};
\draw (9,1) node (d) {$C_{*}(\Sigma)$};
\draw (11,1) node (e) {$0$};
\draw[->](a)--(b);
\draw[->](b)--(c) node [midway,above](TextNode){$\pi_{!}$};
\draw[->](c)--(d) node [midway,above](TextNode){$\pi_{*}$};
\draw[->](d)--(e);
\end{tikzpicture}
\end {equation*}
This exact sequence induces long exact sequences in homology
\begin{equation*}
\begin{tikzpicture}
\draw (1,3) node (a) {$0$};
\draw (3,3) node (b) {$H_{3}(\Sigma, \ML)$};
\draw (6,3) node (c) {$H_{3}(M)$};
\draw (9,3) node (d) {$H_{3}(\Sigma)$};
\draw (11,3) node (e) {};
\draw[->](a)--(b);
\draw[->](b)--(c) node [midway,above](TextNode){$\pi_{!}$};
\draw[->](c)--(d);
\draw[->](d)--(e) node [midway, above](TextNode){$$};
\draw (1,2) node (f) {};
\draw (3,2) node (g) {$H_{2}(\Sigma, \ML)$};
\draw (6,2) node (h) {$H_{2}(M)$};
\draw (9,2) node (i) {$H_{2}(\Sigma)$};
\draw (11,2) node(j) {};
\draw[->](f)--(g);
\draw[->](g)--(h) node [midway,above](TextNode){$\pi_{!}$};
\draw[->](h)--(i);
\draw[->](i)--(j) node [midway, above](TextNode){$$};
\draw (1,1) node (k) {};
\draw (3,1) node (l) {$H_{1}(\Sigma, \ML)$};
\draw (6,1) node (m) {$H_{1}(M)$};
\draw (9,1) node (n) {$H_{1}(\Sigma)$};
\draw (11,1) node (o) {$0$};
\draw[->](k)--(l);
\draw[->](l)--(m) node [midway,above](TextNode){$\pi_{!}$};
\draw[->](m)--(n);
\draw[->](n)--(o);
\end{tikzpicture}
\end {equation*}
\noindent and in cohomology
\begin{equation*}
\begin{tikzpicture}[scale=1.0]
\draw (1,3) node (a) {$0$};
\draw (3,3) node (b) {$H^{1}(\Sigma)$};
\draw (6,3) node (c) {$H^{1}(M)$};
\draw (9,3) node (d) {$H^{1}(\Sigma, \ML)$};
\draw (11,3) node (e) {};
\draw[->](a)--(b);
\draw[->](b)--(c);
\draw[->](c)--(d) node [midway,above](TextNode){$\pi^{!}$};
\draw[->](d)--(e);
\draw (1,2) node (f) {};
\draw (3,2) node (g) {$H^{2}(\Sigma)$};
\draw (6,2) node (h) {$H^{2}(M)$};
\draw (9,2) node (i) {$H^{2}(\Sigma, \ML)$};
\draw (11,2) node(j) {};
\draw[->](f)--(g);
\draw[->](g)--(h);
\draw[->](h)--(i) node [midway,above](TextNode){$\pi^{!}$};
\draw[->](i)--(j);
\draw (1,1) node (k) {};
\draw (3,1) node (l) {$H^{3}(\Sigma)$};
\draw (6,1) node (m) {$H^{3}(M)$};
\draw (9,1) node (n) {$H^{3}(\Sigma, \ML)$};
\draw (11,1) node (o) {$0$};
\draw[->](k)--(l);
\draw[->](l)--(m);
\draw[->](m)--(n) node [midway,above](TextNode){$\pi^{!}$};
\draw[->](n)--(o);
\end{tikzpicture}
\end {equation*}
Combining these with the long exact sequence of the pair $(\Sigma,\ML)$ we obtain the following result.

\begin{prop}
Let $\pi: M \longrightarrow \Sigma$ be a double branched cover over an integral homology sphere $\Sigma$ with branching set a two-component link $\ML$. Then $H_i (M;\Z/2) = H^i (M;\Z/2)$ is isomorphic to $\Z/2$ if $i = 0, 1, 2, 3$, and is zero otherwise.
\end{prop}


\subsection{The cup-product on $H^*(M;\Z/2)$} \label{S:cup}
This section is devoted to the proof of the following result. We continue working with $\Z/2$ coefficients.

\begin{prop}\label{P:cup}
The cup-product $H^1(M) \times H^1(M) \to H^2(M)$ is the bilinear form $\Z/2 \times \Z/2 \to \Z/2$ with the matrix $\lk (\ell_1,\ell_2) \pmod 2$.
\end{prop}

\begin{proof}
We will reduce the cup-product calculation to intersection theory using the commutative diagram
\begin{equation*}
\begin{tikzpicture} [scale =1]
\draw(1,1) node(a) {$H^{1}(M) \times H^{1}(M)$ };
\draw(5,1) node(b) {$H^{2}(M)$};
\draw(1,3) node(c) {$H_{2}({M}) \times H_{2}({M}) $};
\draw(5,3) node(d) {$H_{1}(M)$};
\draw[->](a)--(c) node [midway,left](TextNode){PD};
\draw[->](b)--(d) node [midway,right](TextNode){PD};
\draw[->](a)--(b) node [midway, above] (TextNode) {$\cup$};
\draw[->](c)--(d) node [midway, above] (TextNode) {$\cdot$};
\end{tikzpicture}
\end{equation*}
where $\PD$ stands for the Poincar\'e duality isomorphisms and $\cdot$ for the intersection product. The transfer homomorphism $\pi_!: H_* (\Sigma,\ML) \to H_* (M)$ will give us explicit generators of $H_1(M)$ and $H_2(M)$ that we need to proceed with this approach.

We begin with the group $H_1(M)$. Note that $H_1 (\Sigma,\ML) = \Z/2$ is generated by the homology class $[w]$ of any embedded arc $w \subset \Sigma$ whose endpoints belong to two different components of $\ML$. The transfer homomorphism $\pi_!: H_1 (\Sigma,\ML) \to H_1 (M)$ maps the homology class of $w$ to that of the circle $\pi^{-1}(w)$. Since $\pi_!$ is an isomorphism, we conclude that the circle $\pi^{-1}(w)$ represents a generator of $H_1(M)$.

To describe a generator of $H_2 (M)$, observe that $H_2 (\Sigma,\ML) = \Z/2\,\oplus\,\Z/2$ is generated by the homology classes of Seifert surfaces $S_1$ and $S_2$ of the knots $\ell_1$ and $\ell_2$. We will assume that $S_1$ and $S_2$ intersect transversely in a finite number of circles and arcs, and note that $S_1\,\cap\,S_2$ is homologous to $\lk (\ell_1, \ell_2)\cdot w$. We claim that the closed orientable surfaces $\pi^{-1} (S_1)$ and $\pi^{-1} (S_2)$, representing the homology classes $\pi_!\,([S_1])$ and $\pi_!\, ([S_2])$, are homologous to each other and generate $H_2 (M)$. To see this, we will appeal to Theorem 2 of \cite{LW}, which supplies us with the commutative diagram with an exact row,
\medskip
\begin {equation*}
\begin{tikzpicture}
\draw (1,3) node (a) {$0$};
\draw (4,3) node (b) {$H_3(\Sigma)$};
\draw (7,3) node (c) {$H_2(\Sigma, \ML)$};
\draw (10,3) node (d) {$H_2(M)$};
\draw (13,3) node (e) {$0$};
\draw (7,1) node (f) {$H_{1}(\ML)$};
\draw[->](a)--(b);
\draw[->](b)--(c) node [midway,above](TextNode){$d_*$};
\draw[->](c)--(d) node [midway,above](TextNode){$\pi_{!}$};
\draw[->](d)--(e);
\draw[->](c)--(f) node [midway,right](TextNode){$\partial_{*}$};
\draw[->](b)--(f) node[midway, below,](TextNode){$f \; \;$};
\end{tikzpicture}
\end {equation*}
where $f ([\Sigma]) = [\ell_1] + [\ell_2]$ and $\partial_{*}$ is the connecting homomorphism in the long exact sequence of the pair $(\Sigma, \ML)$. One can easily see that $\partial_{*}$ is an isomorphism. Since $\partial_*([S_1] + [S_2]) = [\ell_1] + [\ell_2] = f([\Sigma])$ we conclude that $[S_1] + [S_2] \in \im d_* = \ker \pi_!$ and hence $\pi_!\,([S_1]) = \pi_!\,([S_2])$ is a generator of $H_2(M)$.

The calculation of the intersection form $H_2(M)\times H_2(M) \to H_1(M)$ is now completed as follows\,:
\begin{multline}\notag
[\pi^{-1}(S_1)]\cdot [\pi^{-1}(S_2)] \,=\, [\pi^{-1} (S_1)\,\cap\,\pi^{-1}(S_2)] \\ =\, [\pi^{-1}(S_1\,\cap\,S_2)]\; =\; \lk(\ell_1,\ell_2)\cdot [\pi^{-1}(w)].
\end{multline}
\end{proof}

\begin{remark}
Let $\beta \in H^1 (M) = \Z/2$ be a generator and assume that $\lk(\ell_1,\ell_2)$ is odd. Proposition \ref{P:cup} implies that $\beta\,\cup\,\beta \in H^2 (M)$ is non-trivial, and a straightforward argument with the Poincar\'e duality shows that $\beta\,\cup\,\beta\,\cup\,\beta$ generates $H^3 (M)$. If $\lk (\ell_1,\ell_2)$ is even then $\beta\,\cup\,\beta = 0$, and the cup-product of $\beta$ with a generator of $H^2 (M)$ generates $H^3 (M)$. This gives a complete description of the cohomology ring $H^* (M)$.
\end{remark}

\begin{example}
The real projective space $\RP^3$ is a double branched cover over the Hopf link in $S^3$ with linking number $\pm 1$. Choose Seifert surfaces $S_1$ and $S_2$ to be the obvious disks intersecting in a single interval $w$. Then $\pi^{-1}(S_1)$ and $\pi^{-1}(S_2)$ are two copies of $\RP^2$, each represented as a double branched cover of a disk with branching set a disjoint union of a circle and a point. These two copies of $\RP^2$ intersect in the circle $\pi^{-1}(w)$ thereby recovering the familiar cup-product structure on $H^*(\RP^3;\Z/2)$.
\end{example}


\end{document}